\documentclass[11pt]{amsart}

\usepackage[cp1250]{inputenc}
\usepackage{graphicx}
\usepackage{url}
\usepackage[all]{xy}
\usepackage{multicol}
\usepackage{amsthm}
\usepackage{amsmath}
\usepackage{amssymb}
\usepackage{geometry}
\usepackage{pinlabel}
\usepackage{hyperref}
\usepackage{array}
\usepackage{xcolor}

\newtheorem{Theorem}{Theorem}[section]
\newtheorem{lemma}[Theorem]{Lemma}
\newtheorem{proposition}[Theorem]{Proposition}
\newtheorem{corollary}[Theorem]{Corollary}
\newtheorem{definition}[Theorem]{Definition}
\newtheorem{remark}[Theorem]{Remark}
\newtheorem{example}[Theorem]{Example}

\newcommand{\RR}{\mathbb {R}}

\newcommand{\PP}{\mathbb {P}}
\newcommand{\tr}{\mathop{\rm{trunk}}}
\newcommand{\K}{\mathcal{K}}
\newcommand{\spin}{\mathcal{S}}
\newcommand{\gr}{\mathop{\textrm{grad}}}
\newcommand{\lay}{\mathop{\rm{lay}}}

\makeatletter
\@namedef{subjclassname@2020}{%
  \textup{2020} Mathematics Subject Classification}
\makeatother

\providecommand{\keywords}[1]
{
  \small	
  \textbf{Keywords :} #1
}

\begin{document}

\title{Flattening knotted surfaces}

\author{Eva Horvat}
\address{University of Ljubljana, Faculty of Education, Kardeljeva plo\v s\v cad 16, 1000 Ljubljana, Slovenia, eva.horvat@pef.uni-lj.si}

\maketitle

\begin{abstract}
A knotted surface in $S^{4}$ may be described by means of a hyperbolic diagram that captures the 0-section of a special Morse function, called a hyperbolic decomposition. We show that every hyperbolic decomposition of a knotted surface $\K $ defines a projection of $\K $ onto a 2-sphere $\Sigma $, whose set of critical values is the hyperbolic diagram of $\K $. We apply such projections, called flattenings, to define three invariants of knotted surfaces: the layering, the trunk, and the partition number. The basic properties of flattenings and their derived invariants are obtained. Our construction is used to study flattenings of satellite 2-knots. 
\end{abstract}


\section {Introduction}
\label{sec1}

Width of knotted surfaces was first considered by Carter and Saito \cite{CS} in the context of charts. A chart of a knotted surface $\K $ in $\RR ^{4}$ is obtained by taking a projection $p\colon \RR ^{4}\to \RR ^{3}$ that is generic with respect to $\K $, and then projecting the singular set of the associated immersion $p(\K )$ onto a 2-plane. Takeda studied the width of surface knots using generic planar projections of embedded surfaces and considering possible images of such projections that consist of fold curves and cusp points \cite{TA}. \\

In this paper, we study width--related invariants of embedded surfaces from a different perspective. Instead of considering generic projections, we apply hyperbolic decompositions of a knotted surface $\K $ to define its layering, trunk and partition number. We show that each hyperbolic decomposition of $\K $ induces a projection of $\K $ to a 2-sphere, and the set of critical values of this projection may be identified with a $ch$-diagram of the surface $\K $, defined by Yoshikawa \cite{YO}. Such diagrams and closely related banded link diagrams are often used in the study of surfaces inside 4-manifolds, thus we hope our description might lead to new results concerning width and related invariants of knotted surfaces. \\

The idea behind our invariants comes from the classical knot theory. The bridge number of classical knots was first considered by Schubert \cite{SC1}, while the more general notion of width was defined by Gabai \cite{GA}. The trunk of 1-knots was defined by Ozawa \cite{OZ}.  
All three invariants might be interpreted in terms of Morse functions as follows \cite{ZU}. Let $K$ be a knot in the 3-sphere $S^3$. Denote by $\mathcal{M}(K)$ the collection of all Morse functions $h\colon S^3\to \RR $ with exactly two critical points, such that the restriction $h|_{K}$ is also Morse. Given a function $h\in \mathcal{M}(K)$, denote by $c_{0}<c_1<\ldots <c_n$ the critical values of $h|_{K}$ and choose regular values $r_i\in (c_{i-1},c_i)$ for $i=1,2,\ldots ,n$. Each function $h$ defines three values 
\begin{xalignat*}{1}
& w(h)=\sum _{i=1}^{n}|K\cap h^{-1}(r_{i})|\;,\quad b(h)=\# \textrm{ of maxima of $h|_{K}$}\;,\quad \tr(h)=\max _{1\leq i\leq n}|K\cap h^{-1}(r_{i})|
\end{xalignat*}
that give rise to three knot invariants: the width $w(K)=\min _{h\in \mathcal{M}(K)}w(h)$, the bridge number $b(K)=\min _{h\in \mathcal{M}(K)}b(h)$ and the trunk of a knot: $\tr (K)=\min _{h\in \mathcal{M}(K)}\tr (h)$. \\

Our aim is to apply a similar construction one dimension higher. First we need to find a suitable family of functions that will do the trick for knotted surfaces. To describe the general setting, we use the following standard results.  

\begin{Theorem} \cite{LEE} \label{th1} Let $M$ and $N$ be smooth manifolds, and let $f\colon M\to N$ be a smooth map with constant rank $k$. Each level set of $f$ is a closed embedded submanifold of codimension $k$ in $M$. 
\end{Theorem}

\begin{lemma}[Ehresmann fibration lemma \cite{EH}] \label{lemma1} Let $M$ and $N$ be smooth manifolds, and let $f\colon M\to N$ be a proper submersion. Then $M$ is a fiber bundle over $N$ with projection given by $f$. 
\end{lemma}

\begin{corollary} \label{cor1} Let $\K$ be a smoothly embedded surface in a smooth 4-manifold $X$ (the embedding being proper at the boundary if necessary), and let $\Sigma $ be a 2-manifold. Consider a smooth map $f\colon X\to \Sigma $  and denote by $A\subset \K$ the set of critical points of $f|_{\K}$. Then cardinality $|\K \cap f^{-1}(x)|$ is constant on each connected component of $\Sigma \backslash f(A)$.  
\end{corollary} 
\begin{proof}
For each regular value of $f|_{\K}$, the fiber is a finite discrete set of points by Theorem \ref{th1}. Moreover, the restriction $f_{\K\backslash A}\colon \K\backslash A\to f(\K\backslash A)\subset \Sigma $ is a fiber bundle by Lemma \ref{lemma1}, so the cardinality of its fibers is constant on each connected component of $\Sigma \backslash f(A)$. 
\end{proof}

Under the setting described in Corollary \ref{cor1}, denote by $U_0,U_1,\ldots ,U_n$ the connected components of $\Sigma \backslash f(A)$ and choose regular values $r_{i}\in U_i$ for $i=0,1,\ldots ,n$. The Corollary implies that the values
\begin{xalignat*}{1}
& \lay (f)=\sum _{i=1}^{n}|\K \cap f^{-1}(r_{i})|\;,\quad p(f)=n\;,\quad  \tr(f)=\max |\K \cap f^{-1}(r_{i})|
\end{xalignat*}
are well defined. Summing up over a suitable collection of maps (or fixing a map and summing up over all equivalent embeddings of $\K$ in $X$), we might arrive at three invariants of knotted surfaces inside $X$. \\

The paper is organized as follows. Section \ref{sec2} contains the basic material about presentations of knotted surfaces that we will need. In Section \ref{sec3}, we present a projection of an embedded surface $\K $ called a flattening, which is associated to a marked graph diagram of $\K $. We discuss its critical values and introduce the terminology needed to describe a flattening. Multiplicities of the flattening map and their basic properties are discussed. In Section \ref{sec4}, we define three invariants of knotted surfaces based on flattenings: the layering, the trunk and the partition number. Some basic results regarding these invariants are obtained. In Subsection \ref{subs41}, we study flattenings and the associated invariants of satellite 2-knots. Subsection \ref{subs42} concludes the paper by offering several ideas for further study. 

\section*{Acknowledgements}

The author would like to thank the anonymous reviewer for helpful comments and suggestions that helped to improve the exposition of the paper and clarify its connections with related studies. The author was supported by the Slovenian Research Agency grants P1-0292 and N1-0083.

\section{Preliminaries}
\label{sec2}

For the remainder of this paper, we restrict our attention to embedded surfaces in $S^4$. Throughout the paper, we denote by $\K $ a smoothly embedded, connected closed surface in $S^4$. In this Section we briefly recall the basic descriptions of embedded surfaces that we will work with. A good introductory overview of the subject is offered in \cite{CS1}. 

\begin{definition} \label{def1} A Morse function $h\colon S^{4}\to \RR $ is called a \textbf{hyperbolic splitting} of an embedded surface $\K$ if it satisfies the following conditions:
\begin{enumerate}
\item $h$ has exactly two critical points on $S^4$,
\item $h_{\K}$ is also Morse,
\item all minima of $h_{\K}$ occur in the level $h^{-1}(-1)$,
\item all maxima of $h_{\K}$ occur in the level $h^{-1}(1)$,
\item all hyperbolic points of $h_{\K}$ occur in the level $h^{-1}(0)$.
\end{enumerate}
\end{definition}
It is well known that every embedded surface in $S^4$ admits a hyperbolic splitting \cite{LO}. We will denote by $\K _{t}=\K \cap h^{-1}(t)$ the $t$-section of the knotted surface, induced by $h$. Similarly, we will denote $S^{4}_{t}=h^{-1}(t)$ and $S^{4}_{I}=h^{-1}(I)$ for any interval $I\subset \RR $. \\

A quite illuminating presentation of an embedded surface may be given by its movie. A \textbf{movie} of a surface $\K $ with a hyperbolic splitting $h\colon S^{4}\to \RR $ is a sequence of diagrams of sections $\K _{t}\subset S^{4}_{t}$ for $t\in [-1,1]$. See Figure \ref{fig3} for a simple movie of a projective plane. 
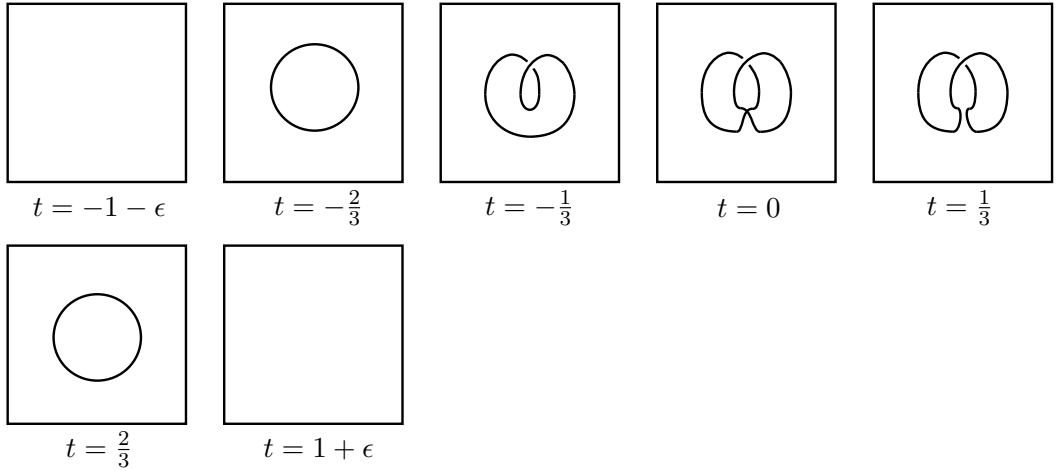
\begin{figure}[h!]
\labellist
\normalsize \hair 2pt
\pinlabel $t=-1-\epsilon $ at 220 540
\pinlabel $t=-\frac{2}{3}$ at 740 540
\pinlabel $t=\frac{1}{3}$ at 2260 540
\pinlabel $t=-\frac{1}{3}$ at 1240 540
\pinlabel $t=0$ at 1760 540
\pinlabel $t=\frac{2}{3}$ at 220 -30
\pinlabel $t=1+\epsilon $ at 740 -30
\endlabellist
\begin{center}
\caption{A movie of a projective plane.}
\includegraphics[scale=0.16]{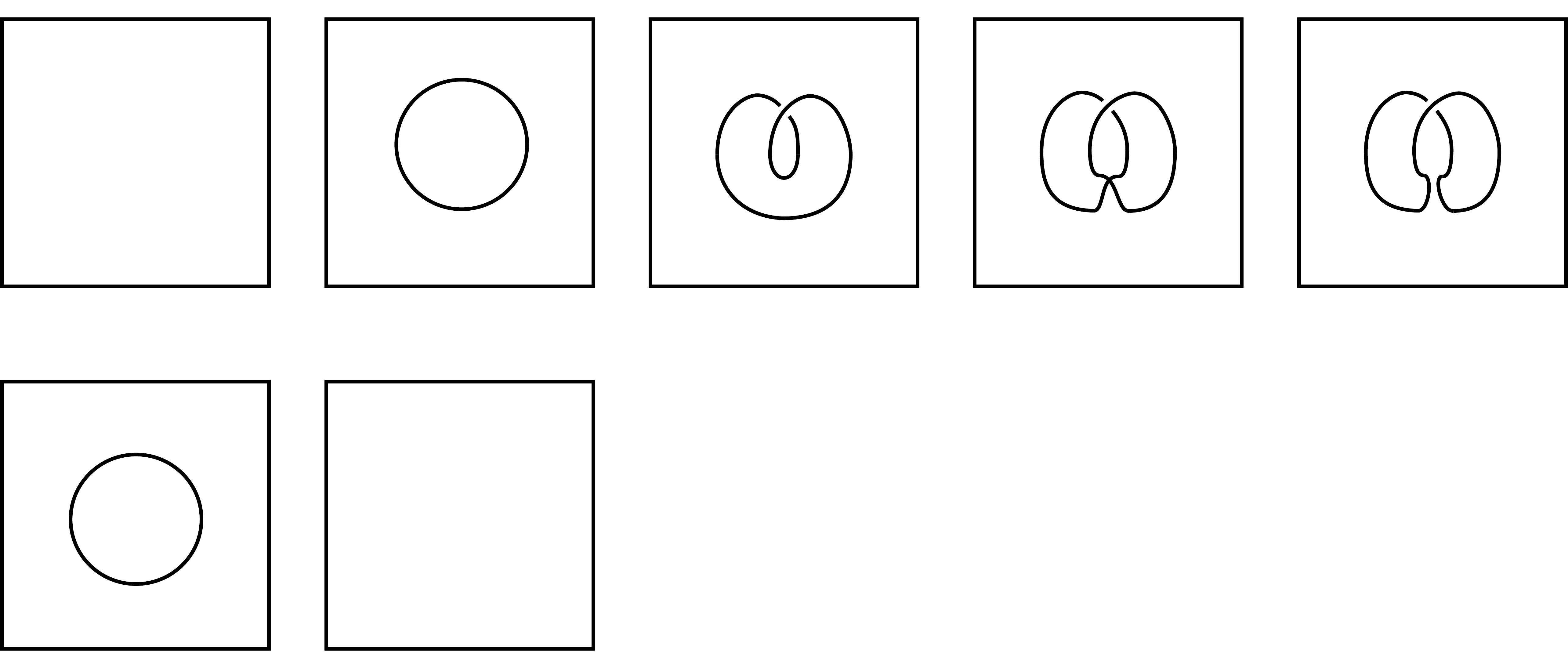}
\label{fig3}
\end{center}
\end{figure}

Instead of following the whole movie, the information about a hyperbolic splitting of a surface may be compressed in a single diagram. Such presentations have been used in the early studies of knotted surfaces (see for example \cite{LO,YO}), and they remain an important tool in the recent literature (see \cite{JA, MZ, HKM}). The basic idea is that a hyperbolic splitting, by definition, induces a handle decomposition of the embedded surface. This handle decomposition may be conveniently presented by either a marked graph or a banded link. \\

Following \cite{MZ}, we define a \textbf{band} for a link $L$ in $S^{3}$ as an embedding $b\colon I\times I\to S^{3}$ such that $b(\partial I \times I)=b\cap L$. We denote a new link $L_{b}=(L\backslash b(\partial I \times I))\cup b(I\times \partial I)$ and call it the link that results from \textbf{resolving} the band $b$. Similarly, if $b$ denotes a collection of pairwise disjoint bands for $L$, we denote by $L_{b}$ the link that results from $L$ by resolving all bands in $b$. A pair $(L,b)$ is called a \textbf{banded link} if $L$ is a link in $S^{3}$, $b$ is a band for $L$ and both $L$ and $L_b$ are unlinks. \\

Let $h\colon S^4\to \RR $ be a hyperbolic splitting of a knotted surface $\K$. It follows from De\-fi\-ni\-tion \ref{def1} that for a small positive $\epsilon $, both $\K _{-\epsilon}$ and $\K _{\epsilon }$ are unlinks. Moreover, at every hyperbolic point of $h_{\K}$, a 1-handle (a band) is added to the boundary of the 0-skeleton of $\K $ along two intervals, embedded in $\K_{-\epsilon }$. The attachement of all 1-handles changes the boundary of the resulting surface, which is obtained from $\K _{-\epsilon }$ by resolving all bands: $\K _{\epsilon }=(\K _{-\epsilon })_{b}$. Thus, the hyperbolic splitting $h$ defines a banded link $(K_{-\epsilon },b)$. \\

Conversely, given a banded link $(L,b)$, the condition that $L$ and $L_{b}$ are unlinks provides a construction of a knotted surface $\K =\K (L,b)$ as follows. View the 4-sphere as $S^{4}=S^{3}\times [-2,2]/(S^{3}\times \{-2\}, S^{3}\times \{2\})$ and let $h\colon S^{4}\to \RR $ be the projection to the second component. Define 
\begin{enumerate}
\item $\K _{-\epsilon }=L$,
\item $\K \cap S^{4}_{[-1,-\epsilon )}$ are disks, capping off every component of $L$,
\item $\K \cap S^{4}_{[-\epsilon ,0)}=L\times [-\epsilon ,0)$,
\item $\K _{0}=L\cup b$, for each component of $b$, add the band $b(I\times I)$ to $L$ along $b(\partial I\times I)$, 
\item $\K \cap S^{4}_{(0,\epsilon ]}=L_{b}\times (0,\epsilon ]$,
\item $\K \cap S^{4}_{(\epsilon ,1]}$ are disks, capping off every component of $L_{b}$.
\end{enumerate}
By \cite[Proposition 2.4]{MZ}, the disks capping off the components of $L$ and $L_b$ are unique up to isotopy. Thus, an embedded surface $\K $ with a hyperbolic splitting $h$ is completely determined by the banded link $(L,b)$, defined by $h$, since $\K =\K (L,b)$. \\

Alternatively, a hyperbolic splitting $h$ of a knotted surface $\K $ may be presented by a \textbf{marked graph diagram}. By Definition \ref{def1}, the 0-section $\K _{0}$ defines an embedded 4-valent graph, with vertices corresponding to saddles (critical points of index 1 of the Morse function $h_{\K }$). A regular projection $p\colon S^{4}_{0}\to \Sigma $ takes $\K _{0}$ to its diagram, a 4-valent graph $\Gamma =p(\K _{0})$ in the 2-sphere $\Sigma $. In this diagram, the vertices corresponding to crossings include the information about the overcrossing and undercrossing strands, while vertices corresponding to saddles are endowed with markers. A marker at a vertex determines the corresponding resolutions below and above the critical point, see Figure \ref{fig0}. We call $\Gamma $ a \textbf{marked graph diagram} of $\K $ with the hyperbolic splitting $h$. Such diagrams were introduced by Yoshikawa \cite{YO}; they are also called ch-diagrams.

\begin{figure}[h!]
\labellist
\normalsize \hair 2pt
\pinlabel $\K _{-\epsilon }$ at 150 -20
\pinlabel $\K _{0}$ at 520 -20
\pinlabel $\K {\epsilon }$ at 890 -20
\endlabellist
\begin{center}
\includegraphics[scale=0.20]{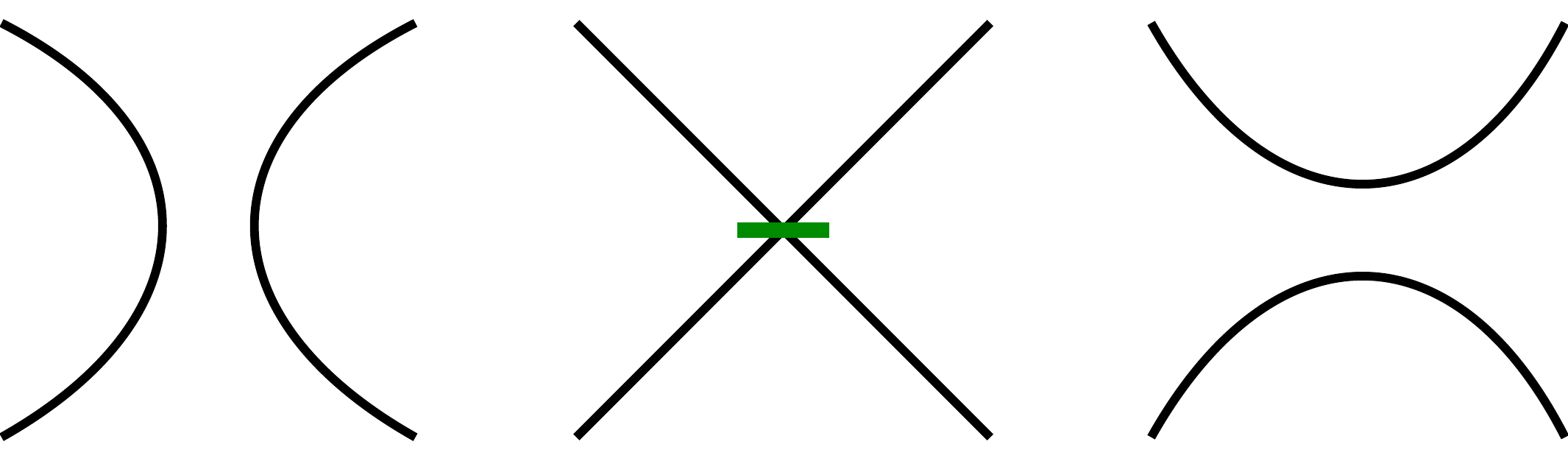}
\caption{The resolutions of a vertex, corresponding to a marker}
\label{fig0}
\end{center}
\end{figure}

By the following theorem, an embedded surface with a hyperbolic splitting is completely defined by its marked graph diagram.
\begin{Theorem}\cite{KSS} Let $\K _i$ be embedded surfaces with hyperbolic splittings $h_{i}$, and let $\Gamma _{i}$ be a marked graph diagram of $h_{i}^{-1}(0)\cap \K _{i}$ for $i=1,2$. If $\Gamma _1=\Gamma _2$, then $\K _1$ is isotopic to $\K _2$.  
\end{Theorem}

\begin{figure}[h!]
\labellist
\normalsize \hair 2pt
\endlabellist
\begin{center}
\includegraphics[scale=0.20]{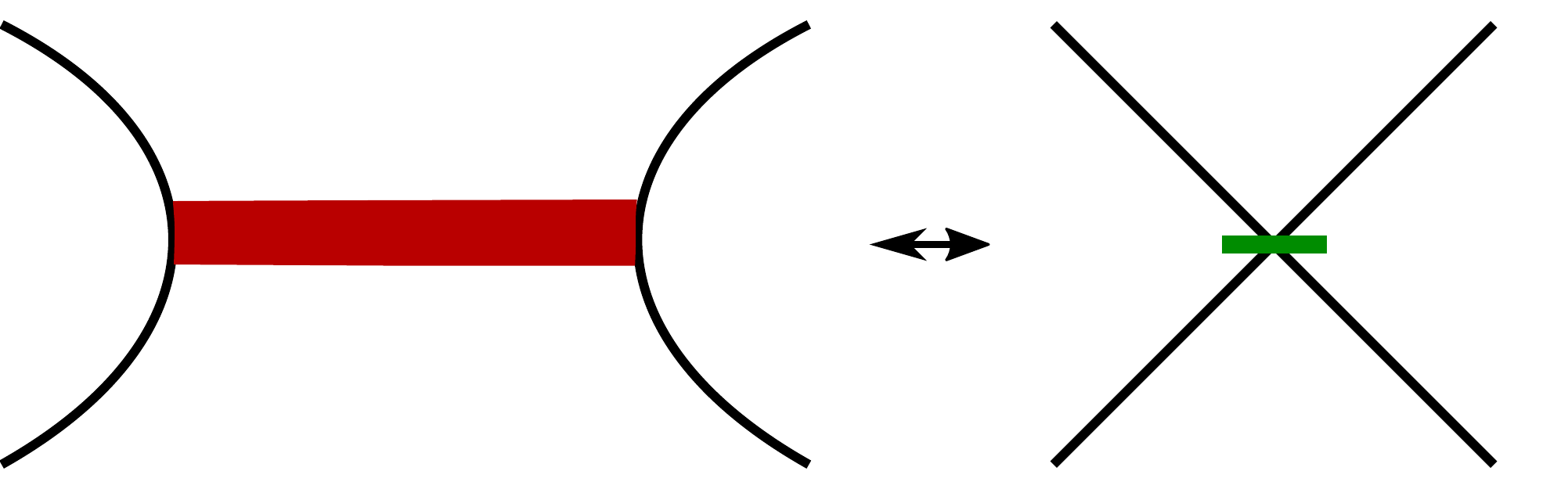}
\caption{The local modification of a banded link that results in a marked graph}
\label{fig01}
\end{center}
\end{figure}

By \cite{JA}, each banded link $(L,b)$ defines a marked graph as follows. Apply an ambient isotopy of $S^3$ to shorten the bands of $b$, until each band is contained in a small disk, then replace the neighborhood of each band with a neighborhood of a marked vertex, see Figure \ref{fig01}. We call the resulting graph the \textit{marked graph, associated with the banded link} $(L,b)$. Conversely, any marked graph $G$ in $S^3$ may be transformed into a banded link by replacing each marked vertex with a band as in Figure \ref{fig01}, and the result is called a \textit{banded link, associated with} $G$. When $G$ is given by a marked graph diagram $\Gamma $, the result of this transformation will also be called a\textit{ banded link, associated with} $\Gamma $.

\section{The flattening of an embedded surface} \label{sec3}

In this Section, we show that each marked graph diagram of an embedded surface $\K $ defines a projection of $\K $ to a 2-sphere. A description of such ``flattened surface'' provides a new perspective on its embedding. \\

To each hyperbolic splitting of a surface $\K $, one may associate two particularly nice families of isotopies of $\K $. Let $j\colon F\to S^{4}$ be a smooth embedding of a surface $\K =j(F)$. Choose a hyperbolic splitting $h$ of $\K $ and let $f\colon F\times I\to S^{4}\times I$ be a smooth isotopy of $\K $ so that $\K =f(F\times \{0\})$. Following \cite{HKM}, we say that $f$ is \textbf{horizontal} with respect to $h$ if $h(pr_1(f(x,t)))$ is independent of $t$ for all $x\in F$. We say that $f$ is \textbf{vertical} with respect to $h$ if for each $x\in F$, the image of $\{x\}\times I$ under $pr_{1}\circ f$ is contained in a single orbit of the flow of $\gr (h)$. Thus, a horizontal isotopy of $\K $ moves $\K _{t}$ within $S^{4}_{t}$, preserving $h_{\K }$. A vertical isotopy of $\K $ changes $h_{\K }$, but preserves the projection of $\K $ onto each level set $S^{4}_{t}$. \\

Let $\Gamma $ be a marked graph diagram of an embedded surface $\mathcal{K}$ in $S^{4}$. It follows from our discussion in Section \ref{sec2} that $\Gamma $ determines a hyperbolic splitting $h\colon S^{4}\to \RR $ of $\mathcal{K}$. Denote by $c(h)$ the two critical points of $h$, and let $v\subset \K _0$ denote the union of all hyperbolic points of $h_{\K }$. Let $p\colon S^{4}_{0}\to \Sigma $ be the projection to a 2-sphere $\Sigma $ which is regular on $\K _{0}\backslash v$ and for which $p(\K _{0})=\Gamma $. Define a projection $h^{\perp} \colon S^4\backslash c(h)\to \Sigma $ as follows. Denote by $\Phi \colon \RR \times S^{4}\to S^{4}$ the flow of the vector field $\gr (h)$. For any point $x\in S^4\backslash c(h)$, set $$h^{\perp}(x)=p\left (\Phi (t,x)\cap S^{4}_{0}\right )\;.$$ In other words, the flow line of the vector field $\gr(h)$ running through a point $x\in S^{4}\backslash c(h)$ intersects the 0-section $S^{4}_{0}$ in a single point; projection of this point onto $\Sigma $ is the image $h^{\perp}(x)$. \\

\begin{lemma} \label{lemma2} The map $h^{\perp}\colon  S^{4}\backslash c(h)\to \Sigma $ is smooth. 
\end{lemma}
\begin{proof} Since $\gr(h)$ is a smooth vector field on a compact smooth manifold $S^4$, it generates a smooth flow $\Phi \colon \RR \times S^{4}\to S^{4}$ by \cite[page 147, Theorem 6]{SP}. For every $x\in S^{4}\backslash c(h)$, there exists a unique value $t_x\in \RR $ such that $h(\Phi (t_x,x))=0$ and thus $h^{\perp}(x)=p(\Phi (t_x,x))$. This defines a smooth map $S^{4}\backslash c(h)\to \RR , x\mapsto t_x$. The projection $p\colon S^{4}_{0}\to \Sigma $ is also smooth. 
\end{proof}

Denote by $h_{\K }$ (resp. $h^{\perp}_{\K }$) the restriction of $h$ (resp. $h^{\perp}$) to $\K $. The map $h^{\perp}_{\K }\colon \K \to \Sigma $ will be called the \textbf{flattening map} that corresponds to the marked graph diagram $\Gamma $ of $\K $. If we think of the hyperbolic splitting as a height function on $\K $, then the flattening map flattens $\K $ against $\Sigma $, and as $\K $ itself is not flat (but might be knotted) we obtain creases along some curves in $\Sigma $. In our case, these creases are simple to describe. 

\begin{proposition} \label{prop1} Let $\Gamma $ be a marked graph diagram of an embedded surface $\K $. If $h$ denotes the hyperbolic splitting of $\K$ defined by $\Gamma $, then the set of critical values of the flattening map $h^{\perp}_{\K} \colon \K \to \Sigma $ equals $\Gamma $. 
\end{proposition}
\begin{proof} It follows from Definition \ref{def1} that the Morse function $h$ has no critical points in a 4-ball containing $\K $. By \cite[page 148, Theorem 7]{SP}, we may choose local  coordinates $(x_{1},x_{2},x_{3},x_{4})$ so that $\gr (h)=\frac{\partial }{\partial x_{4}}$. Thus, the flow lines of $\gr (h)$ are parallel vertical lines. 

The complement of the 0-section $\K _0$ of the surface $\K $ consists of two components $\K _{(0,1]}=\K \cap S^{4}_{(0,1]}$ and $\K _{[-1,0)}=\K \cap S^{4}_{[-1,0)}$. After applying a horizontal isotopy of $\K $ that fixes $\K _0$, we may assume that at every point $x\in \K _{(0,1]}$ (resp. $x\in \K _{[-1,0)}$), the vector field $\gr (h)$ is transverse to $\K $ and thus $x$ has a neighbourhood $\mathcal{U}$ such that $h^{\perp}|_{\mathcal{U}}\colon \mathcal{U}\to h^{\perp}(\mathcal{U})$ is a diffeomorphism. It follows that all critical points of $h^{\perp}_{\K }$ are contained in $\K _0$, and the set of critical values is contained in $\Gamma =p(\K _0)$. 

Let $y\in \Gamma $, then $y=p(x)$ for some $x\in \K _{0}$. First suppose that $x$ is not a hyperbolic point of $h_\K $. Then the gradient flows of $\gr h(x)$ and $\gr (h_{\K })(x)$ coincide and $\gr h(x)$ spans a 1-dimensional linear subspace of $T_{x}\K $ that lies in the kernel of $Dh^{\perp}(x)$, therefore $y=h^{\perp}(x)$ is a critical value of $h^{\perp}_{\K }$. 

In case $x$ is a hyperbolic point of $h_\K $, then $x=\Phi (0,x)$ is a critical point of the projection $p$, since $\K _0$ fails to be a manifold at $x$. Thus $y=h^{\perp}(x)$ is a critical value of $h^{\perp}_{\K }$. 
\end{proof}

In order to describe flattenings of knotted surfaces, we introduce the following ter\-mi\-no\-lo\-gy. Let $\Gamma $ be an embedded 4-valent graph in a 2-sphere $\Sigma $. Each connected component of $\Sigma \backslash \Gamma $ will be called a \textbf{region} of $\Gamma $. Let $\K $ be a smoothly embedded surface in a 4-manifold $M$ (the embedding being proper at the boundary if necessary), and let $f\colon M\to \Sigma $ be a smooth map, such that the set of critical values of $f|_{\K }$ is contained in $\Gamma $. We define the \textbf{multiplicity} of $f|_{\K}$ in a region $U$ as $m_{f|_{\K }}(U)=|\K \cap f^{-1}(x)|$ for any $x\in U$. It follows from Corollary \ref{cor1} that the multiplicity of $f|_{\K }$ in a region is well defined.  \\

 It is often convenient to identify sections $\K _{t}$ for different values of $t$. Denote by $v\subset \K _0$ the set of all hyperbolic points of $h_{\K }$, and let $\mathcal{V}\subset \K $ denote the union of all ascending and descending manifolds of these critical points. For each $t\in (-1,1)\backslash \{0\}$, there exists a diffeomorphism $\rho _{t,0}\colon \K _{t}\backslash \mathcal{V}\to \K _{0}\backslash v$, induced by the gradient flow of the restriction $h_{\K}$. Moreover, the same flow induces a map $\K _{t}\cap \mathcal{V}\to v$ that is two-to-one, and thus the diffeomorphism $\rho _{t,0}$ may be extended to a continuous map $\overline{\rho }_{t,0}\colon \K _{t}\to \K _{0}$. \\

Let $\Gamma $ be a marked graph diagram that defines a hyperbolic splitting $h$ of an embedded surface $\K $. Denote by $\Gamma _{+}$ (resp. $\Gamma _{-}$) the two resolutions of $\Gamma $, defined by the markers: $\Gamma _{-}$ is a diagram of $\K _{-\epsilon }$ and $\Gamma _{+}$ represents a diagram of $\K _{\epsilon }$. The gradient flow of $h_{\K }$ induces diffeomorphisms $\rho _{\pm \epsilon,0}\colon \K _{\pm \epsilon }\backslash \mathcal{V}\to \K _{0}\backslash v$. The diffeomorphism $\rho _{\pm \epsilon ,0}$ may be extended to a continuous map $\overline{\rho }_{\pm \epsilon ,0}\colon \K _{\pm \epsilon }\to \K _{0}$, whose restriction to $\K _{\pm \epsilon }\cap \mathcal{V}$ is two-to-one. This induces maps between the diagrams $\rho _{\pm }\colon \Gamma _{\pm }\to \Gamma $. Diagrams $\Gamma _{-}$ and $\Gamma _{+}$ will be called the ``lower half diagram''  and the ``upper half diagram'' of $\Gamma $.\\

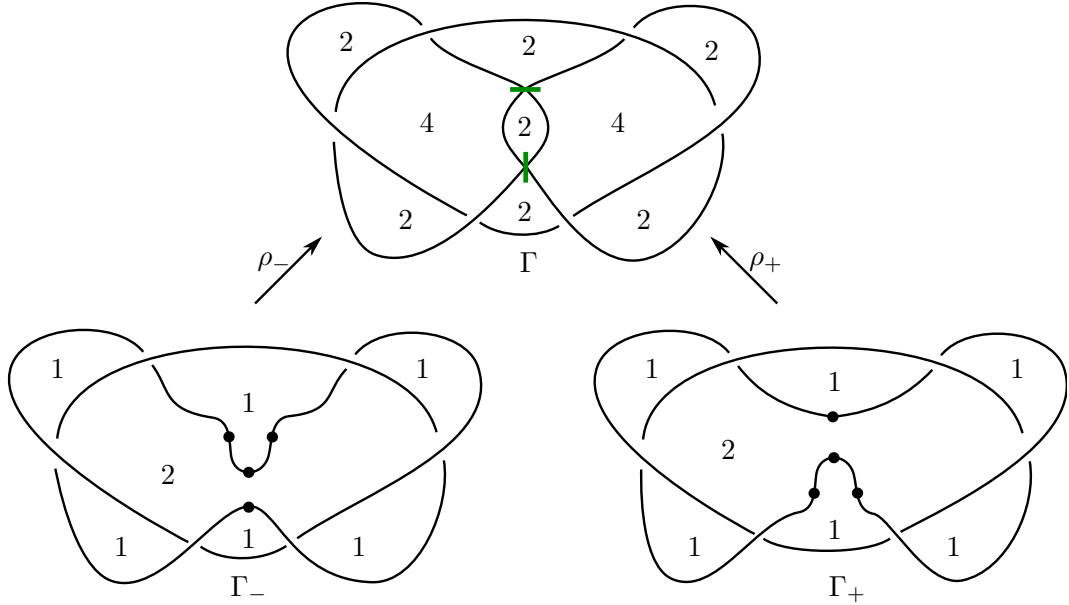
\begin{figure}[h!]
\labellist
\normalsize \hair 2pt
\pinlabel $\Gamma _{-}$ at 570 0
\pinlabel $\Gamma _{+}$ at 1980 0
\pinlabel $\Gamma $ at 1230 770
\pinlabel $1$ at 120 520
\pinlabel $1$ at 980 520
\pinlabel $1$ at 270 100
\pinlabel $1$ at 830 100
\pinlabel $2$ at 380 270
\pinlabel $1$ at 570 120
\pinlabel $1$ at 570 440
\pinlabel $1$ at 1520 520
\pinlabel $1$ at 2380 520
\pinlabel $1$ at 1620 100
\pinlabel $1$ at 2230 100
\pinlabel $1$ at 1950 490
\pinlabel $1$ at 1950 140
\pinlabel $2$ at 1700 330
\pinlabel $2$ at 800 1290
\pinlabel $2$ at 1660 1270
\pinlabel $2$ at 940 870
\pinlabel $2$ at 1500 870
\pinlabel $2$ at 1230 1280
\pinlabel $2$ at 1220 890
\pinlabel $4$ at 990 1100
\pinlabel $4$ at 1440 1100
\pinlabel $2$ at 1220 1090
\pinlabel $\rho _{-}$ at 630 770
\pinlabel $\rho _{+}$ at 1788 770
\endlabellist
\begin{center}
\includegraphics[scale=0.17]{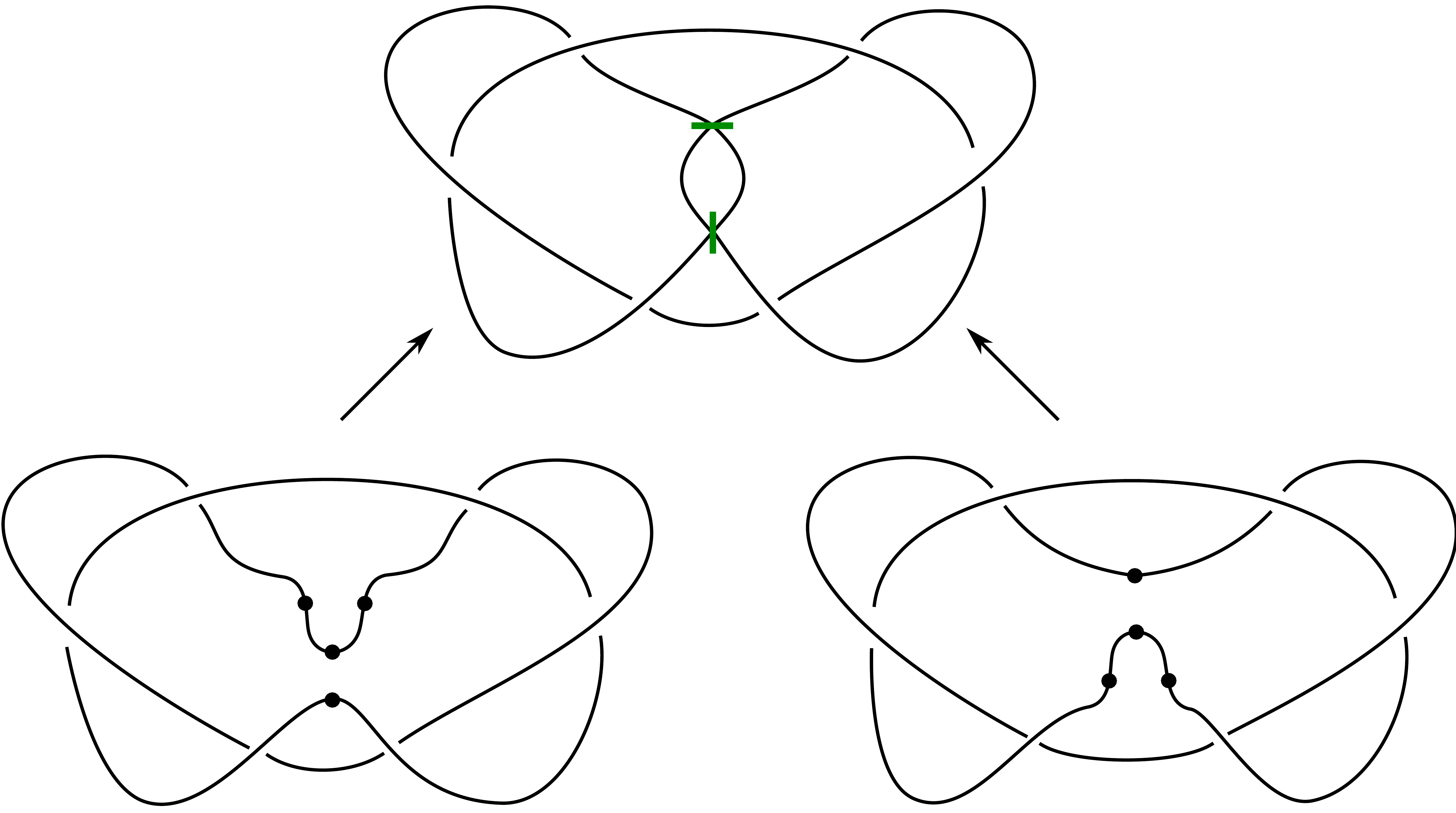}
\caption{A marked graph diagram of the spin of the trefoil knot with its lower and upper half diagrams and nonzero multiplicities of the flattening map in their respective regions}
\label{fig2}
\end{center}
\end{figure}

\begin{example}\label{ex1} Figure \ref{fig2} depicts a marked graph diagram of the spin of the trefoil knot (the construction of spun knots is described on page \pageref{spin}). The hyperbolic splitting $h$, co\-rres\-pon\-ding to this diagram, induces the flattening map $h^{\perp}_{\K }$, whose multiplicities in the respective regions are shown. 
\end{example}

Suppose $U$ is a region of $\Gamma $, then its boundary $\partial U\subset \Sigma $ is a subset of $\Gamma $. Denote by $U^{+}$ the region of $\Gamma _{+}$ for which $\rho _{+}^{-1}(\partial U)\subset \partial U^{+}$. Similarly, denote by $U^{-}$ the region of $\Gamma _{-}$ for which $\rho _{-}^{-1}(\partial U)\subset \partial U^{-}$. We will say that the region $U$ of $\Gamma $ \textbf{is associated with} the region $U^{-}$ of $\Gamma _{-}$ and the region $U^{+}$ of $\Gamma _{+}$. Now $\Gamma _{-}$ (resp. $\Gamma _{+}$) is an unlink, and the disks capping off the components of this unlink represent the 0-handles (resp. 2-handles) of $\K$. Denote $\K ^{+}=\K \cap S^{4}_{[\epsilon ,1]}$ and $\K ^{-}=\K \cap S^{4}_{[-1,-\epsilon ]}$. Applying the same reasoning as in the proof of Proposition \ref{prop1}, the set of critical values of the restriction $h^{\perp}_{\K^{+}}\colon \K^{+}\to \Sigma $ may be identified with $\Gamma _{+}$, while the set of critical values of the restriction $h^{\perp}_{\K^{-}}\colon \K^{-}\to \Sigma $ may be identified with $\Gamma _{-}$. It follows that the multiplicity of $h^{\perp}_{\K }$ in the region $U$ may be computed as 
\begin{xalignat*}{1}
& m_{h^{\perp}_{\K}}(U)=
m_{h^{\perp}_{\K ^{-}}}(U^{-})+m_{h^{\perp}_{\K ^{+}}}(U^{+})\;.
\end{xalignat*}

Multiplicities of the flattening map $h^{\perp}_{\K }$ in the regions of $\Gamma $ are guided by some simple properties that we describe below. Two regions of $\Gamma $ (resp. $\Gamma _{\pm }$) are called \textbf{adjacent} if their boundaries in $\Sigma $ share the same edge of $\Gamma $ (resp. $\Gamma _{\pm }$). 

\begin{lemma}\label{lemma3} Let $\Gamma $ be the marked graph diagram of an embedded surface $\K $ with a hyperbolic splitting $h$. The multiplicities of $h^{\perp}_{\K ^{+}}$ in any two adjacent regions of $\Gamma _{+}$ differ by 1. Also, the multiplicities of $h^{\perp}_{\K ^{-}}$ in any two adjacent regions of $\Gamma _{-}$ differ by 1. 
\end{lemma}
\begin{proof} Let $U$ and $U'$ be two regions of $\Gamma _{+}$, whose boundaries in $\Sigma $ share a common edge $a\subset \Gamma _{+}$. In the handle decomposition of $\K $ induced by $h$, the arc $a$ represents a part of the boundary of a disk that is a 2-handle of $\K $. The addition of this 2-handle causes a splitting into regions along $a$, and an increase of multiplicity in one of these regions (above which the 2-handle lies) by 1. It follows that $m_{h^{\perp}_{\K ^{+}}}(U')=m_{h^{\perp}_{\K ^{+}}}(U)\pm 1$. For the lower half diagram, the proof is analogous.  
\end{proof}

\begin{corollary} \label{cor3} Let $\Gamma $ be the marked graph diagram of an embedded surface $\K $ with a hyperbolic splitting $h$. The multiplicities of the flattening map $h^{\perp}_{\K }$ in any two adjacent regions of $\Gamma $ differ by 0 or 2. 
\end{corollary}
\begin{proof} Let $U_1$ and $U_2$ be two regions of $\Gamma $, whose boundaries in $\Sigma $ share a common edge $a\subset \Gamma $. The region $U_i$ is associated to a region $U_{i}^{+}$ of $\Gamma _{+}$ and to a region $U_{i}^{-}$ of $\Gamma _{-}$ for $i=1,2$. The preimage of $a$ under the map $\rho _{\pm }\colon \Gamma _{\pm }\to \Gamma $ is an edge $\rho ^{-1}_{\pm }(a)$ of $\Gamma _{\pm }$, that is common to the boundaries of $U_{1}^{\pm }$ and $U_{2}^{\pm }$ in $\Gamma _{\pm}$. Thus, $U_{1}^{+}$ and $U_{2}^{+}$ are adjacent regions of $\Gamma _{+}$, and $U_{1}^{-}$ and $U_{2}^{-}$ are adjacent regions of $\Gamma _{-}$. Let us denote $m_{h^{\perp}_{\K ^{+}}}(U_{1}^{+})=k$ and $m_{h^{\perp}_{\K ^{-}}}(U_{1}^{-})=l$, then by Lemma \ref{lemma3} we have $m_{h^{\perp}_{\K ^{+}}}(U_{2}^{+})=k\pm 1$ and $m_{h^{\perp}_{\K ^{-}}}(U_{2}^{-})=l\pm 1$.  It follows that $m_{h^{\perp}_{\K }}(U_{1})=k+l$, while the multiplicity $m_{h^{\perp}_{\K }}(U_{2})$ is either $k+l$ or $k+l\pm 2$. 
\end{proof}

\begin{corollary}\label{cor4} Let $\Gamma $ be the marked graph diagram of an embedded surface $\K $ with a hyperbolic splitting $h$. For any region $U$ of $\Gamma $, the multiplicity $m_{h^{\perp }_{\K }}(U)$ is even. 
\end{corollary}
\begin{proof} This follows directly from the Corollary \ref{cor3} and the fact that every marked graph diagram contains a region where the flattening map has multiplicity 0. 
\end{proof}

Mappings of the plane into the plane were first thoroughly analysed by Whitney, who accomplished that a generic map between two 2-dimensional manifolds may have singular points lying along smooth non-intersecting curves called the ``folds'', and isolated ``cusp'' points on the folds \cite{WH}. In our case, the flattening map $h^{\perp}_{\K }\colon \K \to \Sigma $ is not quite generic, as $\K _0$ contains all hyperbolic points of $h_{\K }$. The marked graph diagram $\Gamma $ represents the image of the fold curves. Its vertices may be of three different types: 
\begin{enumerate}
\item a fold crossing represents a crossing of two fold curves,
\item a branch point where two cusps meet (a non-generic situation),
\item a cusp.
\end{enumerate}

\begin{figure}[h!]
\labellist
\normalsize \hair 2pt
\endlabellist
\begin{center}
\includegraphics[scale=0.2]{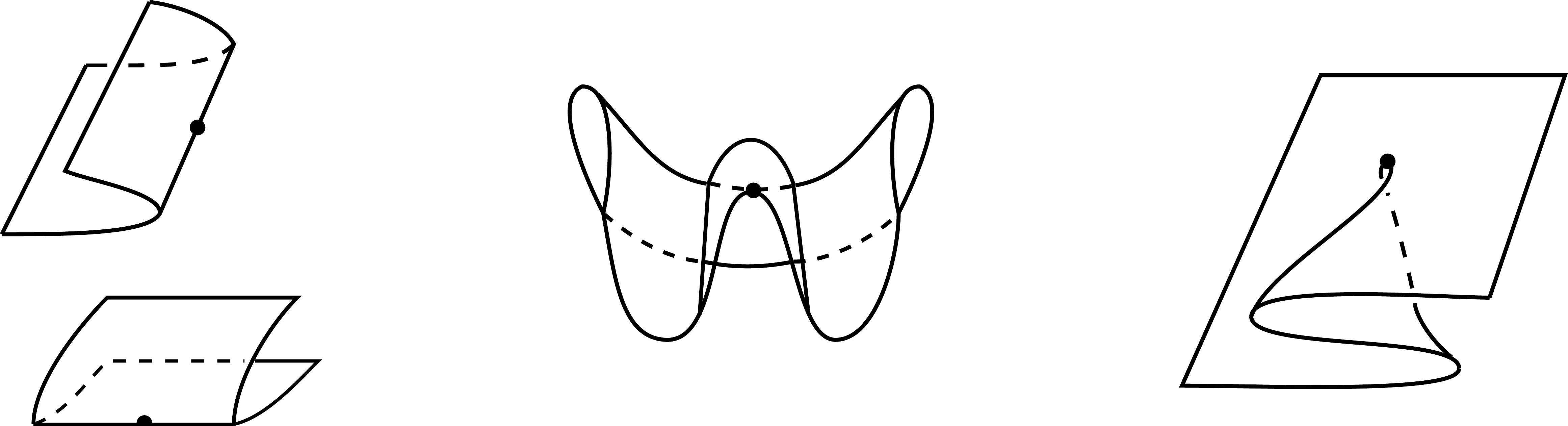}
\caption{The surface above a vertex of $\Gamma $: a fold crossing (left), a branch point (middle) and a cusp (right)}
\label{fig14}
\end{center}
\end{figure}

A part of the surface above each of type of vertex is shown in Figure \ref{fig14}. 
The local multiplicities of the flattening map $h^{\perp}_{\K }$ in the regions surrounding a vertex $v$ of $\Gamma $ are either \begin{itemize}
\item of the form $n$, $n+2$, $n+4$, $n+2$, if $v$ is a fold crossing,
\item of the form $n$, $n+2$, $n$, $n+2$, if $v$ is a branch point, or
\item of the form $n$, $n$, $n$, $n\pm 2$, if $v$ is a cusp
\end{itemize} for some even $n\geq 0$, see Figure \ref{fig13}.

\begin{figure}[h!]
\labellist
\normalsize \hair 2pt
\pinlabel $n+2$ at 140 50
\pinlabel $n$ at 260 150
\pinlabel $n+2$ at 140 260
\pinlabel $n+4$ at 20 150
\pinlabel $n$ at 650 50
\pinlabel $n+2$ at 770 150
\pinlabel $n$ at 650 260
\pinlabel $n+2$ at 530 150
\pinlabel $n\pm 2$ at 1170 50
\pinlabel $n$ at 1300 150
\pinlabel $n$ at 1170 260
\pinlabel $n$ at 1060 150
\endlabellist
\begin{center}
\includegraphics[scale=0.2]{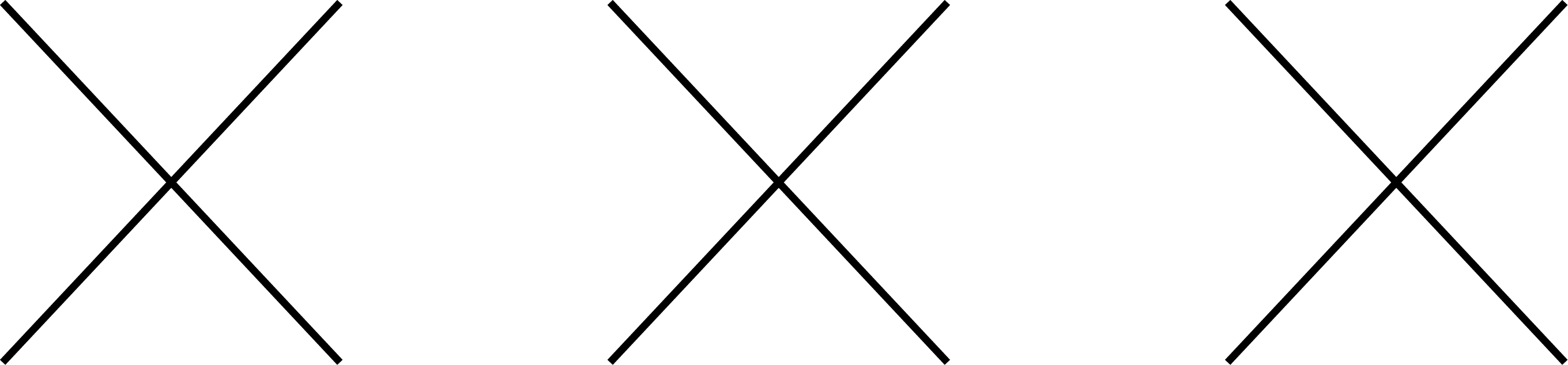}
\caption{Local multiplicities of the flattening map around a vertex of $\Gamma $: a fold crossing (left), a branch point (middle) and a cusp (right)}
\label{fig13}
\end{center}
\end{figure}
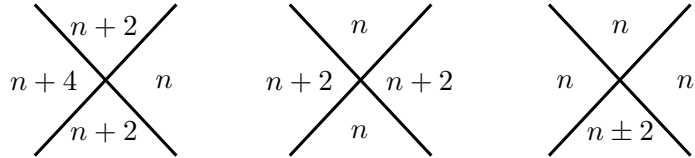

\begin{figure}[h!]
\labellist
\normalsize \hair 2pt
\pinlabel $\Gamma _{-}$ at 60 0
\pinlabel $\Gamma _{+}$ at 1680 0
\pinlabel $\Gamma $ at 800 420
\pinlabel $1$ at 140 220
\pinlabel $2$ at 300 250
\pinlabel $1$ at 1190 220
\pinlabel $2$ at 710 680
\pinlabel $2$ at 860 680
\pinlabel $2$ at 1000 680
\endlabellist
\begin{center}
\includegraphics[scale=0.15]{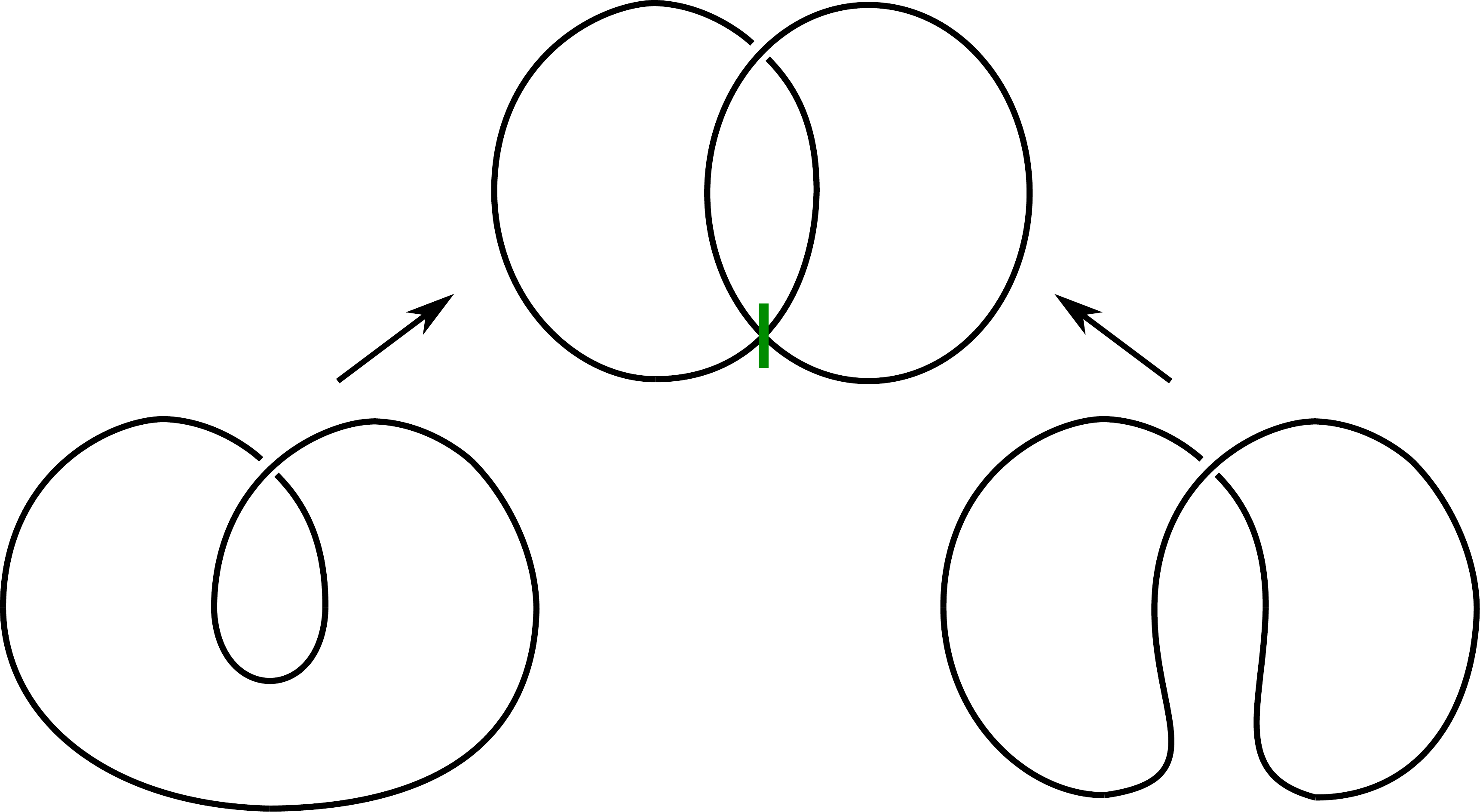}
\caption{A marked graph diagram of the projective plane with its lower and upper half diagrams}
\label{fig4}
\end{center}
\end{figure}
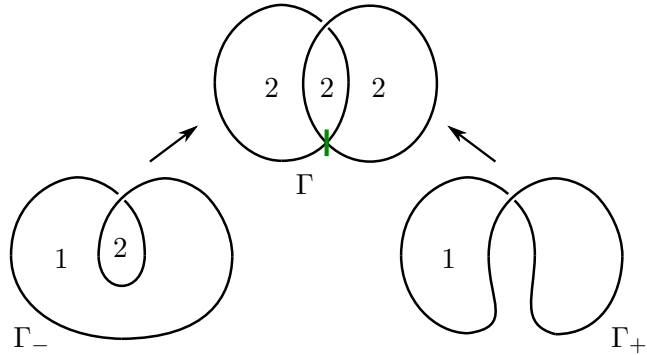

\begin{example}\label{ex2} In Figure \ref{fig4}, a marked graph diagram of the projective plane is given. Both vertices of the diagram $\Gamma $ represent cusps of the flattening map $h^{\perp}_{\K }$. 
\end{example}

Multiplicities of the map $h^{\perp}_{\K ^{+}}$ (resp. $h^{\perp}_{\K ^{-}}$) in the regions of the upper (resp. lower) half diagram of $\K $ may also be determined by considering the movie of the knotted surface $\K $. A movie of $\K $ captures diagrams of sections $\K _{t}\subset S^{4}_{t}$ for $t\in [-1,1]$. Two successive diagrams in this sequence will differ at most by a Reidemeister move of type I, II or III, the birth or death of some simple closed curves (at $t=\pm 1$), or some saddle points (at $t=0$). Figure \ref{fig6} depicts the change of local multiplicities of $h^{\perp}_{\K ^{\pm }}$ during each of these interactions. 

\begin{figure}[h!]
\labellist
\tiny \hair 2pt
\pinlabel $k-1$ at 210 380
\pinlabel $k$ at 210 80
\pinlabel $k$ at 210 600
\pinlabel $k-1$ at 210 820
\pinlabel $k$ at 215 275
\pinlabel $k$ at 630 700
\pinlabel $l$ at 750 700
\pinlabel $m$ at 880 700
\pinlabel $k$ at 600 300
\pinlabel  \rotatebox{90}{$k-l+m$} at 750 220
\pinlabel  \rotatebox{90}{$(\ast )$} at 800 220
\pinlabel $m$ at 910 300
\pinlabel $l$ at 750 30
\pinlabel $l$ at 750 390
\pinlabel $0$ at 1830 340
\pinlabel $1$ at 1830 210
\pinlabel $0$ at 1830 700
\pinlabel $k$ at 2370 820
\pinlabel $k\pm 1$ at 2370 700
\pinlabel $k$ at 2370 580
\pinlabel $k$ at 2370 210
\pinlabel \rotatebox{90}{$k\pm 1$} at 2240 210 
\pinlabel \rotatebox{90}{$k\pm 1$} at 2510 210
\pinlabel $k$ at 1130 700
\pinlabel $k$ at 1130 220
\pinlabel \rotatebox{90}{$k\pm 1$} at 1247 542
\pinlabel \rotatebox{90}{$k\pm 1$} at 1247 100
\pinlabel $l$ at 1230 850
\pinlabel $m$ at 1340 850
\pinlabel\rotatebox{90}{$m\pm 1$} at 1440 700
\pinlabel $l$ at 1230 370
\pinlabel $m$ at 1340 370
\pinlabel \rotatebox{90}{$m\pm 1$} at 1450 220
\pinlabel $l\pm 1$ at 1345 220
\pinlabel $(\ast )$ at 1240 680
\pinlabel \rotatebox{90}{$(\ast )\pm 1$} at 1345 580
\endlabellist
\begin{center}
\includegraphics[scale=0.16]{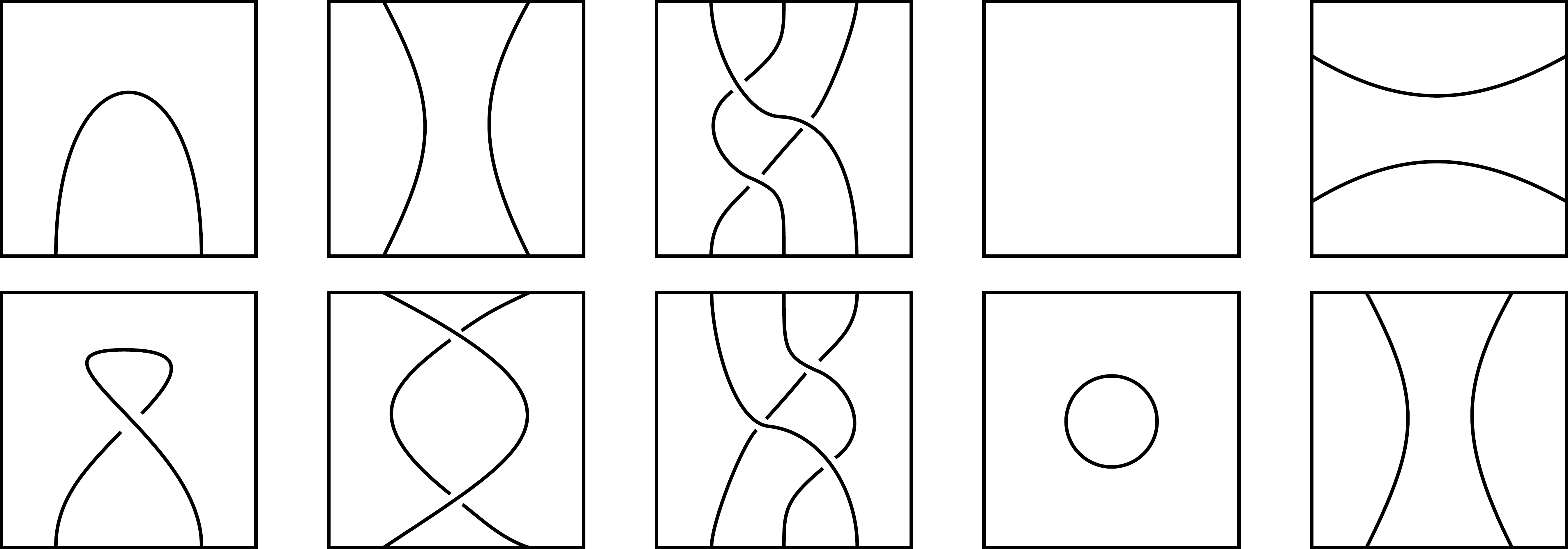}
\caption{Change of the local multiplicities of the map $h^{\perp}_{\K ^{\pm }}$ during ele\-men\-ta\-ry string interactions of movies, defined in \cite{CS}. From left to right: Reidemeister moves I, II, III, the birth of a simple closed curve, and a saddle.}
\label{fig6}
\end{center}
\end{figure}
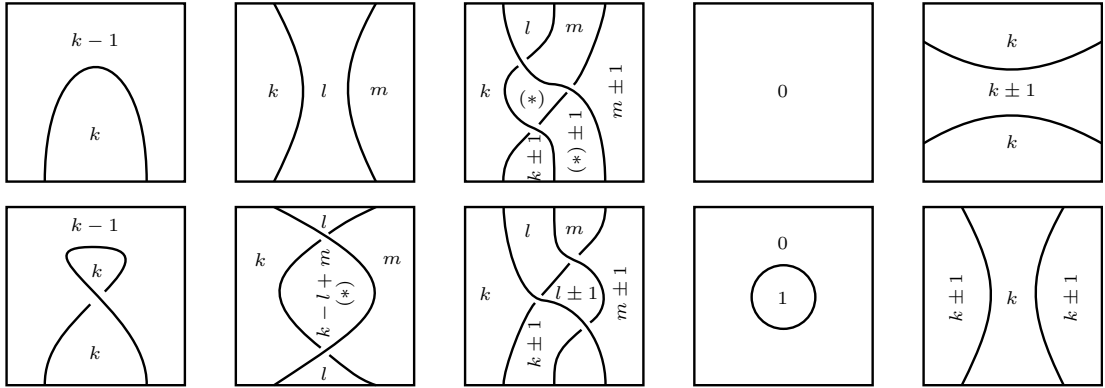

\begin{lemma} \label{lemma4} Let $\Gamma $ be a marked graph diagram of a closed surface $\K $, smoothly embedded in $S^{4}$. Denote by $h$ the hyperbolic splitting of $\K $, defined by $\Gamma $. The multiplicity of $h^{\perp }_{\K }$ in any region $U$ of $\Gamma $ is completely determined by the marked graph diagram $\Gamma $. 
\end{lemma}
\begin{proof} The region $U$ is associated with a region $U^{+}$ of $\Gamma _{+}$ and with a region $U^{-}$ of $\Gamma _{-}$. The resolution $\Gamma _{+}$ is the diagram of an unlink $\K _{\epsilon }$ in the 3-sphere $S^{4}_{\epsilon }$. By \cite[Proposition 2.4]{MZ}, the collection of 2-disks, capping off the components of this unlink, is unique up to isotopy. We may thus choose any sequence of moves from Figure \ref{fig6} to deform $\Gamma _{+}$ into a diagram of a split unlink. Tracing the change of multiplicities backwards in this sequence yields the multiplicity $m_{h^{\perp }_{\K ^{+}}}(U^{+})$. The same reasoning may be applied on the resolution $\Gamma _{-}$ to obtain the multiplicity $m_{h^{\perp }_{\K ^{-}}}(U^{-})$, and adding up both, we obtain $m_{h^{\perp }_{\K }}(U)$. 
\end{proof}

\subsection{Relationship between flattenings and generic planar projections}
Takeda studied embedded surfaces in $\RR ^4$ by using generic projections to the plane \cite{TA}. In this Subsection, we discuss the relationship between his perspective and flattenings, coming from hyperbolic splittings.  

\begin{definition} \label{defT} \cite{TA} Let $f\colon F\to \RR ^4$ be an embedding of a closed connected surface and let $\pi \colon \RR ^{4}\to \RR ^{2}$ be an orthogonal projection. We say that $\pi $ is \textbf{generic} with respect to $f$ if $\pi \circ f$ is a $C^{\infty }$ stable mapping.   
\end{definition}

\begin{proposition} \label{prop6} \cite{TA} Let $f\colon F\to \RR ^2$ be a smooth mapping of a closed connected surface to the plane. Denote by $S(f)$ the set of singular points of $f$. Then $f$ is $C^{\infty }$ stable iff $S(f)$ consists merely of fold points and cusps, if its restriction to the set of fold points is an immersion with normal crossings and if for each cusp $q$ we have $$f^{-1}(f(q))\cap S(f)=\{q\}\;.$$ 
\end{proposition}

If $f\colon F\to S^4$ is an embedding of a surface with a hyperbolic decomposition $h\colon S^{4}\to \RR $, then the flattening $h^{\perp }\colon S^{4}\to \Sigma $ does not represent a generic projection with respect to $f$. However, the flattening map can be slightly disturbed to obtain a generic projection. 

\begin{proposition} \label{prop7} Let $f\colon F\to \RR ^{4}$ be an embedding of a closed connected surface and denote by $j\colon \RR ^{4}\to S^{4}$ the compactification map. For each marked graph diagram $\Gamma $ of $(j\circ f)(F)$, there exists a projection $\pi \colon \RR ^{4}\to \RR ^{2}$ that is generic with respect to $f$, so that the set of critical values of $\pi \circ f$ is isotopic to the upper half diagram $\Gamma _{+}$. 
\end{proposition}
\begin{proof}  Denote $\mathcal{K}=(j\circ f)(F)$ and let $\Gamma $ be a marked graph diagram of $\K $ that defines a hyperbolic splitting $h\colon S^{4}\to \RR $ of $\mathcal{K}$. Choose a small positive number $0<r<1$. Let $p\colon S^{4}_{r}\to \RR ^{2}$ be the projection which is regular on $\mathcal{K}_{r}$ and for which $p(\mathcal{K}_{r})=\Gamma _{+}$. Denote by $\Phi \colon \RR \times S^{4}\to S^{4}$ the flow of the vector field $\gr (h)$ and define a projection $\widehat{h}\colon S^{4}\backslash c(h)\to \RR ^{2}$ by $$\widehat{h}(x)=p\left (\Phi (t,x)\cap S^{4}_{r}\right )\;,$$ where $c(h)$ denotes the two critical points of $h$. Using the same reasoning as in the proofs of Lemma \ref{lemma2} and Proposition \ref{prop1}, we show that $\widehat{h}$ is a smooth projection whose set of critical values equals $\Gamma _{+}$. 

Choose a 4-ball neighborhood $U$ of $\mathcal{K}$ so that $\mathcal{K}\subset U\subset S^{4}\backslash c(h)$ and a diffeomorphism $\psi \colon U\to \RR ^{4}$. After applying a horizontal isotopy of $\K $ if neccessary, we may assume that the set of singular points of the composition $\widehat{h}\circ \psi ^{-1}\circ f\colon F\to \RR ^{2}$ consists merely of fold points immersed with normal crossings. It follows by Proposition \ref{prop6} that $\widehat{h}\circ \psi ^{-1}\circ f$ is $C^{\infty }$ stable and its set of critical values is isotopic to $\Gamma _{+}$. Therefore, $\pi =\widehat{h}\circ \psi ^{-1}\colon \RR ^{4}\to \RR ^{2}$ is generic with respect to $f$. 
\end{proof}

\section{Invariants of embedded surfaces that arise from flattenings}\label{sec4}

In this Section, we apply flattenings to define three invariants of knotted surfaces. Denote by $\mathcal{G}(\K )$ the collection of all marked graph diagrams of an embedded surface $\K$. Each marked graph diagram $\Gamma \in \mathcal{G}(\K)$ defines a hyperbolic splitting $h\colon S^{4}\to \RR $ and a smooth flattening map $h^{\perp }\colon S^{4}\backslash c(h) \to \Sigma $, where $\Sigma $ denotes a 2-sphere inside the 0-section $S^{4}_{0}$. By Proposition \ref{prop1}, the set of critical values of $h^{\perp}_{\K}$ equals $\Gamma $. A vertex of $\Gamma $ is called \textbf{inessential} if it is a marked vertex that represents a branch point of the flattening map $h^{\perp}_{\K }$. Any vertex of $\Gamma $ that is not inessential is called \textbf{essential}. Two regions $U$ and $U'$ of $\Gamma $ will be called \textbf{equivalent} if there exists a chain of regions $U_0=U,U_1,U_2\ldots ,U_k=U'$ such that the boundaries $\partial U_i$ and $\partial U_{i+1}$ in $\Sigma $ share the same inessential vertex and their associated regions in $\Gamma _{+}$ coincide: $U_{i}^{+}=U_{i+1}^{+}$ in $\Gamma _{+}$ for $i=0,1,\ldots ,k$. It is easy to see this defines an equivalence relation on the set of regions of $\Gamma $. Moreover, in two equivalent regions, the flattening map $h^{\perp}_{\K }$ has the same multiplicity, see Figure \ref{fig17}. 

\begin{figure}[h!]
\labellist
\normalsize \hair 2pt
\pinlabel $n\pm 2$ at 140 40
\pinlabel $n$ at 240 135
\pinlabel $n\pm 2$ at 140 240
\pinlabel $n$ at 40 135
\endlabellist
\begin{center}
\includegraphics[scale=0.2]{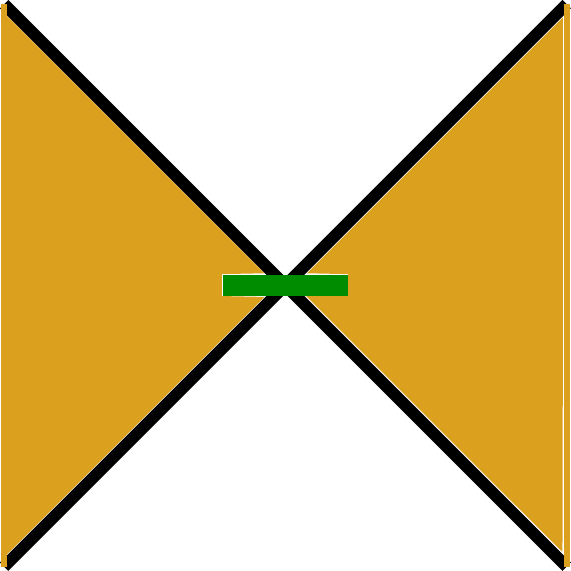}
\caption{In two equivalent regions, the flattening map $h^{\perp}_{\K }$ has the same multiplicity}
\label{fig17}
\end{center}
\end{figure}
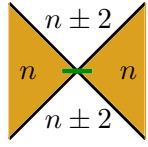

Denote by $U_{0},U_{1},\ldots ,U_{n}$ the regions of $\Gamma $. Define an equivalence relation $\sim $ on the index set $\{0,1,\ldots ,n\}$ by $$i\sim j\Leftrightarrow U_i\textrm{ is equivalent to }U_j\;.$$ Denote by $[i]$ the equivalence class of an index $i$ and let $\mathcal{I}=\left \{[i]|\, i\in \{0,1,\ldots ,n\}\right \}$. Now define  
\begin{xalignat*}{1}
& \lay (\Gamma )=\sum _{[i]\in \mathcal{I}}m_{h^{\perp}_{\K}}(U_{i})\;,\;\; \tr(\Gamma )=\max _{[i]\in \mathcal{I}}m_{h^{\perp}_{\K}}(U_i)\;, \; \; p(\Gamma )=\# \left \{[i]\in \mathcal{I}|\,m_{h^{\perp}_{\K}}(U_i)>0\right \}\, \;,\\
& \lay (\K )=\min _{\Gamma \in \mathcal{G}(\K)}\lay (\Gamma )\;,\quad \tr(\K )=\min _{\Gamma \in \mathcal{G}(\K)}\tr (\Gamma )\;,\quad p(\K )=\min _{\Gamma \in \mathcal{G}(\K)}p(\Gamma )
\end{xalignat*}
The values in the last line will be called the \textbf{layering} of $\K $, the \textbf{trunk} of $\K $ and the \textbf{partition number} of $\K $ respectively. Clearly, these invariants are related to some extent:

\begin{proposition} \label{prop3} For any smoothly embedded closed surface $\K $ we have 
\begin{xalignat}{1}\label{eq1}
& \lay (\K )\geq 2p(\K )+\tr (\K)-2\;.
\end{xalignat}
\end{proposition}
\begin{proof} Let $\Gamma \in \mathcal{G}(\K )$ be any marked graph diagram of $\K $, and let $h^{\perp }_{\K }$ be the corresponding flattening map. By Corollary \ref{cor4}, any nonzero multiplicity of $h^{\perp }_{\K }$ in a region of $\Gamma $ is $\geq 2$, thus the sum over all equivalence classes of regions gives $$2(p(\Gamma )-1)+\tr (\Gamma )\leq \lay(\Gamma )\;.$$ Now choose a marked graph diagram $\Gamma _1\in \mathcal{G}(\K )$ for which $\lay(\Gamma _1)=\lay(\K )$, then $\lay(\K )\geq 2(p(\Gamma _1)-1)+\tr (\Gamma _1)\geq 2p(\K )+\tr (\K )-2$. 
\end{proof}

\begin{remark}\label{rem1} Takeda defined similar invariants of embedded surfaces using generic projections to the plane \cite{TA}. Let $f\colon F\to \RR ^4$ be an embedding of a closed connected surface and let $\pi \colon \RR ^{4}\to \RR ^2$ be an orthogonal projection that is generic with respect to $f$. Then the set of singular points $S(\pi \circ f)$ consists of folds and cusps, and $(\pi \circ f)(S(\pi \circ f))$ divides the plane into several regions. The value $|(\pi \circ f)^{-1}(x)|$ for an element $x$ in a given region is called the \textit{local width}, while $w(f,\pi )$ denotes the maximum of the local widths over all the regions and $tw(f,\pi )$ denotes the sum of the local widths over all the regions. The \textbf{width} $w(f(F))$ of an embedded surface $f(F)$ is the minimum of $w(\widehat{f},\widehat{\pi })$, where $\widehat{f}$ runs over all the embeddings isotopic to $f$ and $\widehat{\pi }$ runs over all orthogonal projections which are generic with respect to $\widehat{f}$. The \textbf{total width} $tw(f(F))$ of an embedded surface $f(F)$ is the minimum of $tw(\widehat{f},\widehat{\pi })$, where $\widehat{f}$ runs over all the embeddings isotopic to $f$ and $\widehat{\pi }$ runs over all orthogonal projections which are generic with respect to $\widehat{f}$. 
\end{remark}

\begin{corollary} \label{cor2} For any embedded surface $\mathcal{K}$ in $S^{4}$, its width (as defined by Takeda) and its trunk are related by $w(\mathcal{K})\leq \tr (\K )$. Moreover, its total width (as defined by Takeda) and its layering are related by $tw(\mathcal{K})\leq \lay (\K )$. 
\end{corollary}
\begin{proof} Let $f\colon F\to S^{4}$ be an embedding with $f(F)=\mathcal{K}$. Suppose that $\tr (\K )=k$ and let $\Gamma $ be a marked graph diagram with $\tr (\Gamma )=k$. Denote by $h\colon S^{4}\to \RR $ the hyperbolic splitting of $\K $, associated with the marked graph $\Gamma $. By Proposition \ref{prop7}, there exists a projection $\pi \colon \RR ^{4}\to \RR ^{2}$ that is generic with respect to $f$, such that the set of critical values of $\pi \circ f$ is isotopic to $\Gamma _{+}$. Each region $U$ of $\Gamma $ is associated with a region $U^{+}$ of $\Gamma _{+}$ and the multiplicity of the flattening map $m_{h|_{\K }^{\perp}}(U)$ equals the local width of the projection $\pi \circ f$ in the region $U^{+}$. Moreover, for any two regions $U_{1}$ and $U_2$ we have $U_{1}^{+}=U_{2}^{+}$ if and only if $U_{1}\sim U_{2}$. It follows that $w(f,\pi )=k$ and consequently $w(\K )\leq \tr (\K )$. 

Similarly, let $\lay (\K )=m$ and let $\Gamma $ be a marked graph diagram with $\lay (\Gamma )=m$. Denote by $h\colon S^{4}\to \RR $ the hyperbolic splitting of $\K $, associated with the marked graph $\Gamma $. By Proposition \ref{prop7}, there exists a projection $\pi \colon \RR ^{4}\to \RR ^{2}$ that is generic with respect to $f$, such that the set of critical values of $\pi \circ f$ is isotopic to $\Gamma _{+}$. By similar reasoning as in the previous paragraph, it follows that $tw(f,\pi )=m$ and thus $tw(\K )\leq \lay (\K )$. 
\end{proof}

Let us consider the simplest class of embedded surfaces: those which are unknotted. Recall that an orientable surface $\K $ in $S^{4}$ is \textbf{unknotted} if it bounds a handlebody. By \cite[Theorem 1.2]{HK}, a surface $\K $ in $S^{4}$ is unknotted if and only if it is isotopic to a surface in $S^{3}\subset S^{4}$.

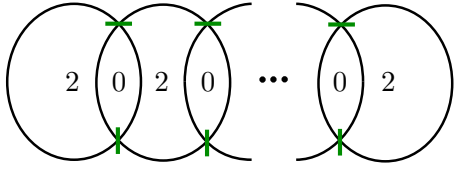
\begin{figure}[h!]
\labellist
\normalsize \hair 2pt
\pinlabel $2$ at 180 220
\pinlabel $0$ at 305 220 
\pinlabel $2$ at 420 220 
\pinlabel $0$ at 545 220 
\pinlabel $0$ at 900 220 
\pinlabel $2$ at 1030 220 
\endlabellist
\begin{center}
\includegraphics[scale=0.15]{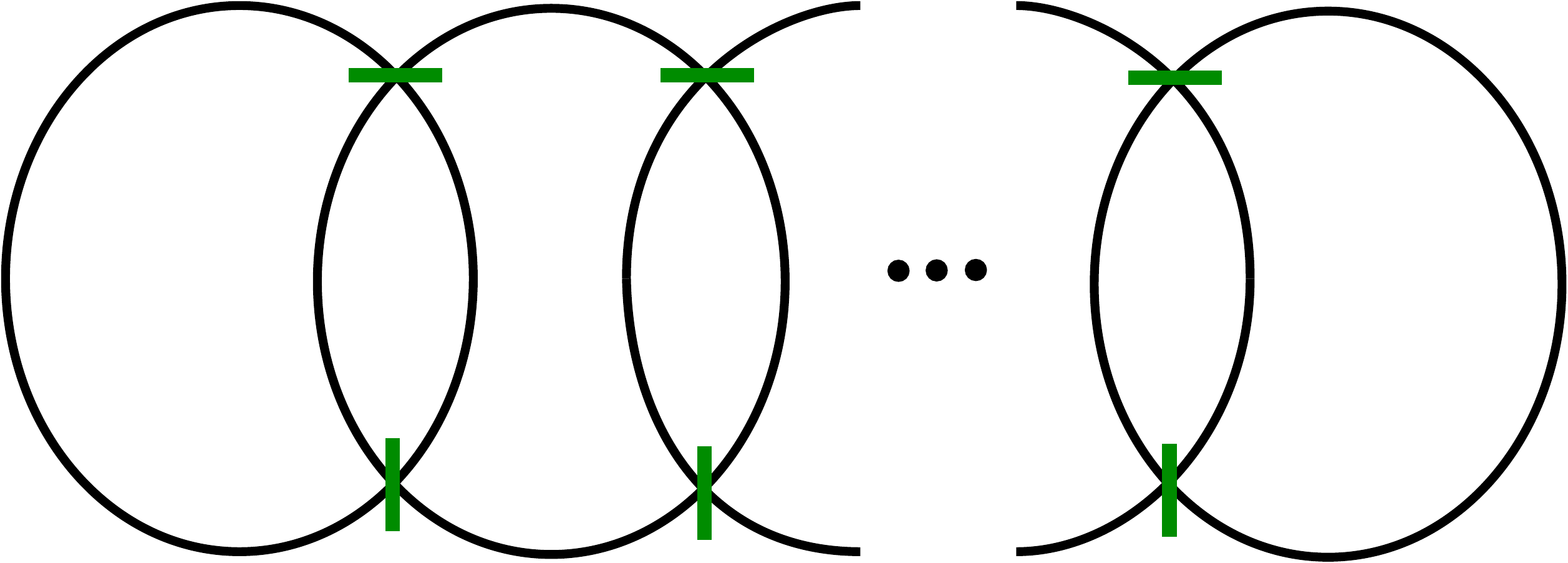}
\caption{A marked graph diagram of the unknotted orientable surface of genus $g$ (the number of circles equals $g+1$)}
\label{fig7}
\end{center}
\end{figure}

\begin{lemma} \label{lemma7} Let $\K $ be an orientable closed surface in $S^4$. The following statements are equivalent:
\begin{enumerate}
\item[(i)] $\K $ is unknotted. 
\item[(ii)] $p(\K )=1$
\item[(iii)] $\lay (\K )=2$
\item[(iv)] $\tr(\K )=2$
\end{enumerate}
\end{lemma}
\begin{proof}
$(i)\Rightarrow (ii)$ If $\K $ is an unknotted sphere, it admits a marked graph diagram without any vertices (a circle), thus $p(\K )=1$. Suppose $\K $ is an unknotted orientable closed surface of genus $g>0$, then it admits a marked graph diagram $\Gamma $ with $2g$ inessential vertices, see Figure \ref{fig7}. All regions of $\Gamma $ in which $h^{\perp}_{\K }$ has nonzero multiplicity, belong to the same equivalence class, thus $p(\Gamma )=1$ and consequently $p(\K )=1$. \\
$(ii)\Rightarrow (iii)$ Suppose $p(\K )=1$. Then $\K $ admits a marked graph diagram in which all regions where $h^{\perp}_{\K }$ has nonzero multiplicity belong to the same equivalence class, and by Corollary \ref{cor3} this multiplicity equals $2$. It follows that $\lay (\K )=2$. \\
$(iii)\Rightarrow (iv)$ The implication is obvious. \\
$(iv)\Rightarrow (i)$  By Corollary \ref{cor2}, an embedded orientable surface $\K $ with $\tr (\K )=2$ has $w(K)=2$, and it follows by \cite[Theorem 3.3]{TA} that $\K $ is unknotted. 
\end{proof}

Next, we examine how our invariants behave under connected sum of surfaces. 

\begin{proposition} \label{prop4} Let $\K _1$ and $\K _2$ be closed connected smoothly embedded surfaces in $S^{4}$, then 
\begin{xalignat*}{1}
& p(\K _1 \# \K_2)\leq p(\K _1)+p(\K _2)-1\;,\quad  \lay (\K _1 \# \K_2)\leq \lay (\K _1)+\lay (\K _2)-2\;,\\
& \textrm{ and }\tr (\K _1 \# \K_2)\leq \max \{\tr (\K _1),\tr (\K _2)\}\;.
\end{xalignat*}
\end{proposition}
\begin{proof} Let $\Gamma _i \in \mathcal{G}(\K _i)$ be a marked graph diagram of the embedded surface $\K _i$, and let $h_i $ be their corresponding hyperbolic splittings for $i=1,2$. Choose a 2-disk $B_i$ that contains $\Gamma _i$ for $i=1,2$, then disjointly embed these disks into a common 2-sphere $\Sigma $ by a map $j\colon B_1 \sqcup B_2\to \Sigma $. Denote by $U$ the common region of $j(\Gamma _1)$ and $j(\Gamma _2)$, then $\Sigma \backslash (j(B_1)\cup j(B_2))\subset U$. Choose a region $U_i$ of $j(\Gamma _i)$ that is adjacent to $U$ and has $m_{(h_{i}^{\perp})_{\K _i}}(j^{-1}(U_i))=2$ for $i=1,2$. Choose two arcs $a_i \subset \partial U_i\cap \partial U$ for $i=1,2$ and join the regions $U_1$ and $U_2$ by adding a band along $a_1\cup a_2$, then replace this band with a marked vertex as in Figure \ref{fig01} to obtain the connected sum of graphs $\Gamma =j(\Gamma _1 )\# j(\Gamma _2)$. Then $\Gamma $ represents a marked graph diagram for $\K _{1} \# \K _{2}$. Every region of $\Gamma $ where $h^{\perp }_{\K _{1}\# \K _{2}}$ has nonzero multiplicity is either a region of $j(\Gamma _1)$, a region of $j(\Gamma _2)$ or the region coming from $U_1$ and $U_2$, therefore $p(\Gamma )=p(\Gamma _1)+p(\Gamma _2)-1$. If $W$ is any region of $\Gamma _{i}$, different from $U_1$ and $U_2$, then $m_{h^{\perp}_{\K _1\# \K _2}}(j(W))=m_{h_{i}^{\perp}|_{\K _i}}(W)$. The multiplicity of $h^{\perp }_{\K _{1} \# \K _{2}}$ in the region arising from $U_1$ and $U_2$ equals 2. It follows that $\lay (\Gamma )=\lay (\Gamma _1)+\lay (\Gamma _2)-2$ and $\tr (\Gamma )=\max \{\tr (\Gamma _1),\tr (\Gamma _2)\}$. 

Thus, for every pair of marked graph diagrams $\Gamma _{i}\in \mathcal{G}(\K _i)$, there exists a marked graph diagram $\Gamma \in \mathcal{G}(\K _{1} \# \K _{2})$ such that $p(\Gamma )=p(\Gamma _1)+p(\Gamma _2)-1$, $\lay (\Gamma )=\lay (\Gamma _1)+\lay (\Gamma _2)-2$ and $\tr (\Gamma )=\max \{\tr (\Gamma _1),\tr (\Gamma _2)\}$. Choosing the diagrams $\Gamma _i$ so that $p(\Gamma _1)=p(\K _1)$ and $p(\Gamma _2)=p(\K _2)$, it follows that $p(\K _1 \# \K _2)\leq p(\K _1)+p(\K _2)-1$. Si\-mi\-lar\-ly, we may conclude that $\lay (\K _1 \# \K _2)\leq \lay (\K _1)+\lay (\K _2)-2$ and $\tr (\K _1 \# \K _2)\leq \max \{\tr (\K _1),\tr (\K _2)\}$. 

\end{proof}

\begin{remark} \label{rem2} Observe that the inequality \eqref{eq1} from Proposition \ref{prop3} is an equality in the case of an unknotted orientable surface $\K$ by Lemma \ref{lemma7}. Likewise, all three inequalities from Proposition \ref{prop4} are equalities when both $\mathcal{K}_1$ and $\mathcal{K}_2$ are unknotted orientable surfaces.
\end{remark}

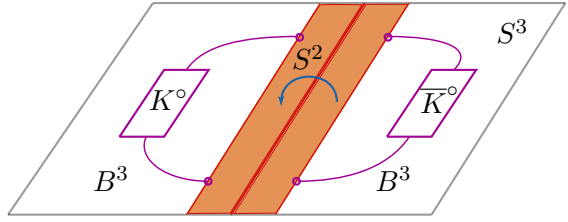
\begin{figure}[h!]
\labellist
\normalsize \hair 2pt
\pinlabel $K^{\circ }$ at 434 315
\pinlabel $\overline{K}^{\circ }$ at 1164 315 
\pinlabel $B^{3}$ at 280 100
\pinlabel $B^{3}$ at 1040 100
\pinlabel $S^{2}$ at 810 430
\pinlabel $S^{3}$ at 1360 500
\endlabellist
\begin{center}
\includegraphics[scale=0.15]{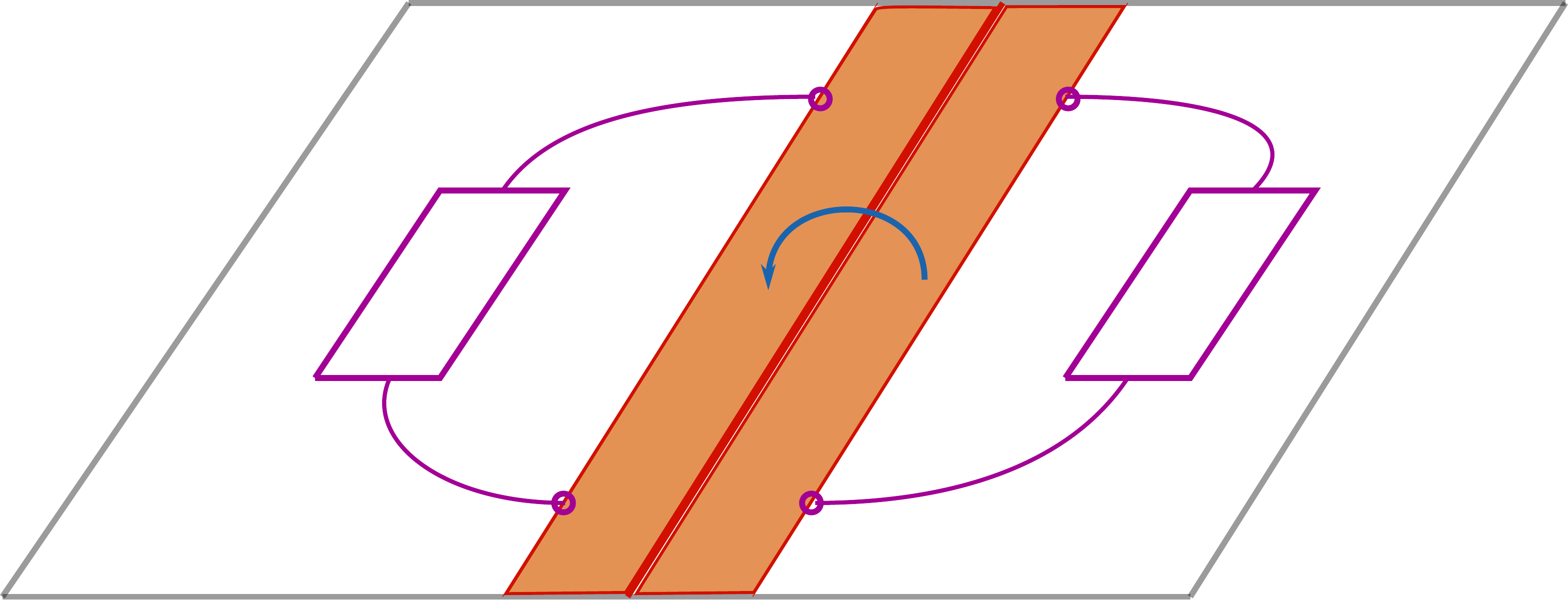}
\caption{The construction of spun knots}
\label{fig16}
\end{center}
\end{figure}

\label{spin}
Spun knots, introduced by Artin, represent the oldest known examples of knotted spheres \cite{AR}. Let $K$ be a 1-knot in $S^3$. A 1-tangle $(B^3,K^{\circ})$ with endpoints on two antipodal points $p_1,p_2\in \partial B^{3}$ is obtained by removing a small open ball neighborhood of a point on $K$. View the 4-sphere as $S^{4}=B^{4}\cup _{S^{3}}B^{4}$ and decompose the equatorial 3-sphere as $S^{3}=B^{3}\cup _{S^{2}}B^{3}$. Let $(B^{3},K^{\circ })$ be a 1-tangle in the first 3-ball. Choose a regular neighborhood $\nu S^{2}\cong S^{2}\times B^{2}$, then spin the pair $(B^{3},K^{\circ })$ around $S^{2}\times \{0\}$ in the complement of $\nu S^{2}$. This gives a decomposition $$(B^{3},K^{\circ })\times S^{1}\cup (S^{2},\{p_0,p_1\})\times B^{2}=(S^{4},\mathcal{S}(\K ))\;.$$ Capping off the annulus $K^{\circ }\times S^{1}$ by the two disks $\{p_0,p_1\}\times B^{2}$, we obtain the knotted sphere $\spin (K)$, called the \textbf{spin} of $K$. Since our definitions of width, trunk and partition number of embedded surfaces originate from similar invariants of 1-dimensional knots, an obvious question arises whether the values of 1-dimensional invariants of a knot $K$ are in any way connected with the corresponding 2-dimensional invariants of the surface knot $\spin (K)$.

\begin{Theorem} \label{th3} Let $K$ be a 1-knot with bridge number $b(K)$. Then $$\tr (\spin (K))\leq 2b(K)\;.$$
\end{Theorem}
\begin{proof} Suppose $K$ is a 1-knot with $b(K)=k$. Then $K$ may be given as the plat closure of a braid $\beta $ on $2k$ strands. In \cite{MZ}, the authors constructed a banded link diagram for $\spin (K)$ that is shown on the left of Figure \ref{fig9}. Denote by $\Gamma \in \mathcal{G}(\spin (K))$ the corresponding marked graph diagram of $\K $, depicted in the right of Figure \ref{fig9}. Let $\Sigma $ be the 2-sphere, containing the diagram $\Gamma $. 

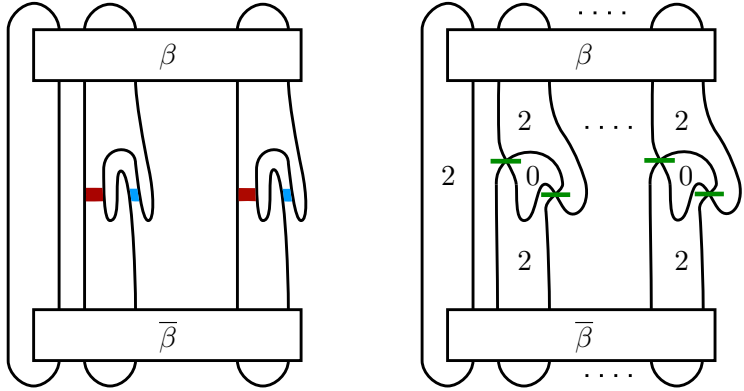
\begin{figure}[h!]
\labellist
\normalsize \hair 2pt
\pinlabel $\overline{\beta }$ at 380 120
\pinlabel $\beta $ at 380 780 
\pinlabel $\overline{\beta }$ at 1360 120
\pinlabel $\beta $ at 1360 780 
\pinlabel $2$ at 1040 500
\pinlabel $2$ at 1220 300 
\pinlabel $2$ at 1220 630
\pinlabel $2$ at 1590 300 
\pinlabel $2$ at 1590 630
\pinlabel $0$ at 1240 500
\pinlabel $0$ at 1600 500
\endlabellist
\begin{center}
\includegraphics[scale=0.17]{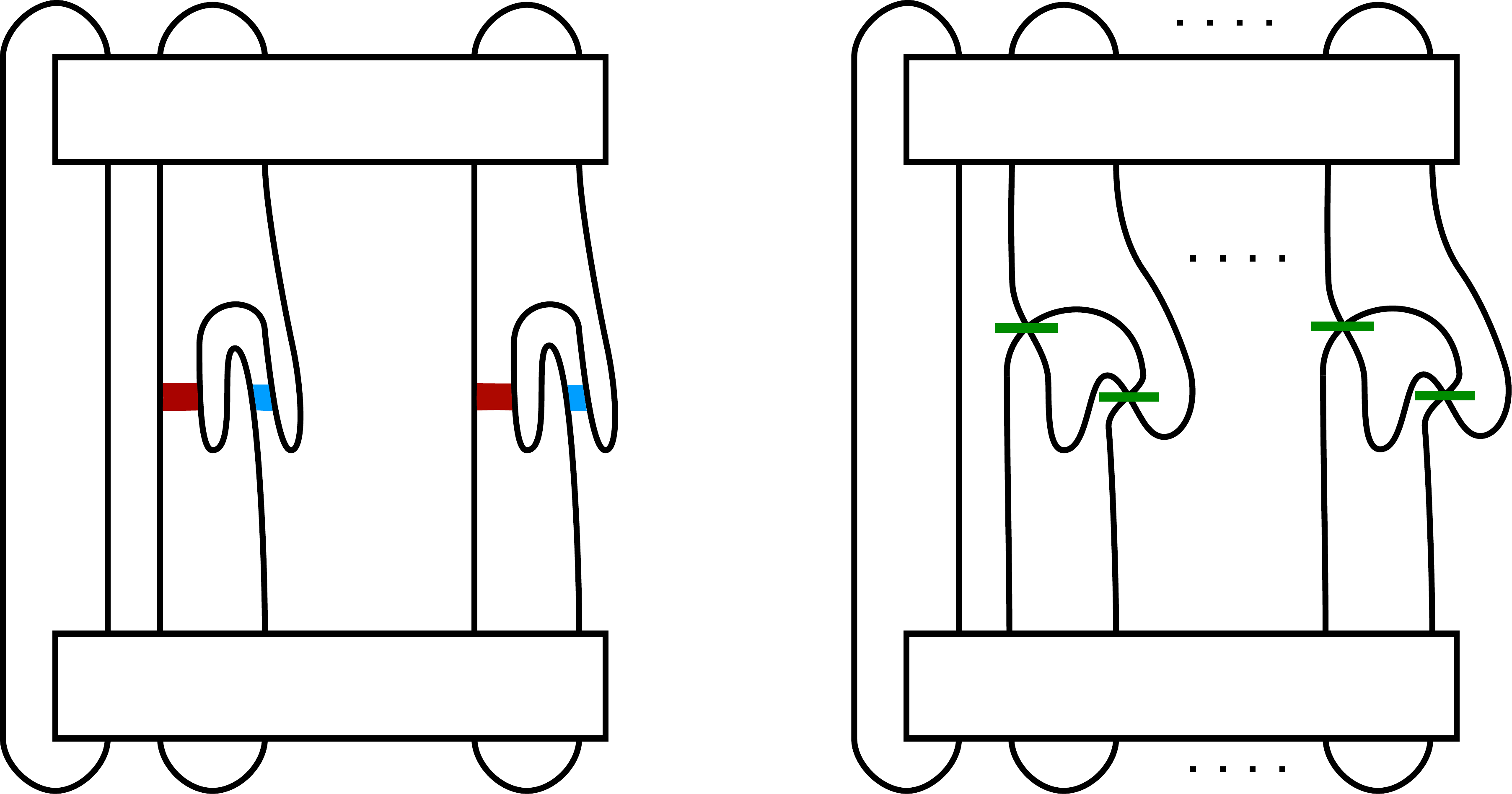}
\caption{A banded link diagram (left) and the marked graph diagram (right) of the spin $\spin (K)$, where $K$ is given as the plat closure of a braid $\beta $}
\label{fig9}
\end{center}
\end{figure}

Choose any region $U$ of the diagram $\Gamma $ where the flattening map $h^{\perp}_{\spin (K)}$ has nonzero multiplicity. Since $\beta $ is a braid on $2k$ strands, there exists a horizontal path $\alpha \colon I\to \Sigma $ from $x\in U$ to a point $y\in V$, where $V$ is a region of $\Gamma $ where $h^{\perp}_{\spin (K)}$ has multiplicity zero, so that $\alpha $ crosses $\Gamma $ at most $k$ times. By Corollary \ref{cor3}, the multiplicities of $h^{\perp}_{\spin (K)}$ in any two adjacent regions of $\Gamma $ differ by at most 2, therefore $m_{h^{\perp}_{\spin (K)}}(U)\leq 2k$. It follows that $\tr (\Gamma )\leq 2k$ and consequently $\tr (\spin (K))\leq 2k$. 
\end{proof}

\begin{corollary} \label{cor5} If $K$ is a 2-bridge knot, then $\tr (\mathcal{S}(K))=4$. 
\end{corollary}
\begin{proof} By \cite[Proposition 3.8]{TA}, an $n$-twist spun 2-bridge knot $\K $ has $w(\K )=4$ for any $n\neq \pm 1$. When $n=0$, Corollary \ref{cor2} implies that $\tr (\mathcal{S}(K))\geq w(\mathcal{S}(K))=4$ and by Theorem \ref{th3}, the equality follows. 
\end{proof}

Beside topological properties of flattenings, such as partition number and region multiplicities, one might consider their geometric properties, such as the shape of regions and the number of vertices of given type. For a flattening $h^{\perp}_{\K }\colon \K \to \Sigma $ corresponding to a marked graph diagram $\Gamma $, the closure of each region in $\Sigma $ is a curved edge polygon, whose vertices are the vertices of $\Gamma $. For any region $U$ of $\Gamma $, denote by $\widehat{U}=\bigcup _{U_i\sim U} U_i$ the subset of $\Sigma $ containing all regions equivalent to $U$, and let $s(U)$ denote the number of essential vertices in $\partial \widehat{U}$. 
Denoting by $U_{0},U_{1},\ldots U_{n}$ the regions of $\Gamma $ and by $\mathcal{I}=\left \{[i]|\, i\in \{0,1,\ldots ,n\}\right \}$ the set of equivalence classes of indices, we may define $$s(\Gamma )=\frac{\sum _{[i]\in \mathcal{I}}s(U_i)}{p(\Gamma )}\quad \textrm{and}\quad s(\K )=\min _{\Gamma \in \mathcal{G}(\K )}s(\Gamma )\;.$$ We call this invariant the \textbf{shape} of a knotted surface $\K $.  Shape distinguishes the unknotted orientable surfaces: 

\begin{proposition}\label{prop5} Let $\K $ be an orientable closed surface in $S^4$. Then $\K $ is unknotted if and only if $s(\K )=0$. 
\end{proposition}
\begin{proof} If $\K $ is an unknotted surface of genus $g$, it admits a marked graph diagram $\Gamma $ in Figure \ref{fig7} with $2g$ marked vertices. Since all the vertices in this diagram are inessential, we have $s(\Gamma )=0$ and consequently $s(\K )=0$.  

Suppose $\K $ is a surface with $s(\K )=0$, then there exists a marked graph diagram $\Gamma \in \mathcal{G}(\K )$ with no essential vertices. It follows that $\Gamma $ has no crossings (all its vertices are saddles), thus $\K _{-\epsilon }$ and $\K _{\epsilon }$ are unlinks without crossings. In the section $S^{4}_{-\epsilon }$ of the hyperbolic decomposition corresponding to $\Gamma $, we may cap off the components of $\K _{-\epsilon }$ by disks, add the bands that correspond to 1-handles and obtain a surface whose boundary is an unlink without crossings (equivalent to $\K _{\epsilon }$), and may thus be capped off by disks inside the same section. We obtain a surface $\K ^{\prime}\subset S^{4}_{-\epsilon }$ that is isotopic to $\K $ by \cite[Proposition 2.4]{MZ}. Since $\K $ is isotopic to a surface inside $S^{3}$, it is unknotted.  
\end{proof}

\subsection{Flattenings of satellite 2-knots} \label{subs41}

Flattenings of surfaces are a useful tool for the study of satellite 2-knots. Let $\K _{P}$ be a 2-sphere embedded in $S^{2}\times D^{2}$, and let $\K _C$ be a 2-sphere embedded in $S^{4}$ with a tubular neighborhood $\nu (\K _C)$. If $f\colon S^{2}\times D^{2}\to \nu (\K _C)$ is a diffeomorphism, then $f(\K _P)$ is called a \textbf{satellite knot} with \textbf{pattern} $\K _P$ and \textbf{com\-pa\-nion} $\K _{C}$. We recall the construction of banded link diagrams of satellite knots, described in \cite{HKM}. \\

Suppose $\K $ is a satellite knot with pattern $\K _P$ and companion $\K _C$. View the 4-sphere as $S^{4}=S^{3}\times [-2,2]/(S^{3}\times \{-2\}, S^{3}\times \{2\})$, with a Morse function $h\colon S^{4}\to \RR $ that projects to the second factor. Choose an embedding of $\K _C$ for which $\K _{C}\cap S^{4}_{0}$ is a 1-knot, while the saddles of $\K _{C}$ lie in the sections $S^{4}_{\pm 1}$; this is the normal form of \cite{KSS}. 

For the ambient manifold of the pattern, we choose $V=S^{2}\times D^{2}\subset S^{4}$ which intersects the 0-section in a solid torus $W=V\cap S^{4}_{0}\cong S^{1}\times D^{2}$, while $V\cap S^{4}_{[-2,0]}=V\cap S^{4}_{[0,2]}\cong D^{2}\times D^{2}$. Draw a banded link diagram for $\K _{P}$ inside $W$ (that lies in the 0-section $S^{4}_{0}$). Choose a meridian disk $\mathcal{D}$ of $W$ that is disjoint from all bands in the diagram for $\K _P$; the number $\omega $ of (unsigned) intersection points $\K _P\cap \mathcal{D}$ is called the \textbf{geometric winding} of $\K _P$ (this number depends on the choice of $\mathcal{D}$).  

We draw a banded link diagram for $\K _C$ and move it by isotopy so that it lies inside $W'=\nu (\K _{C})\cap S^{4}_{0}$. Choose a meridian disk $\mathcal{D}'$ for $W'$ that intersects the banded link transversely in one point, then isotope the diffeomorphism $f\colon V\to \nu (\K _C)$ so that $f(W)=W'$ and $f(\mathcal{D}\times I)=W'\backslash (\mathcal{D}'\times I)$. A banded link diagram for $\K $ is obtained by drawing the 0-framed satellite of $\K _{P}\cap S^{4}_{0}\subset W$ around $\K _{C}\cap S^{4}_{0}$, attaching the bands corresponding to $\K _{P}$ and attaching $\omega $ copies of each band corresponding to $\K _{C}$ (the bands that lie below $S^{4}_{0}$ need to be pushed above, which is done by taking their dual bands). 

\begin{example} \label{ex3} Let $\K _{C}$ be the spin of the figure eight knot, whose marked graph diagram is given in the middle of Figure \ref{fig10}. Take a simple pattern $\K _{P}$ that winds around the 2-sphere $S^{2}\times \{0\}$ in $S^{2}\times D^{2}$ three times; its marked graph diagram is given on the left of Figure \ref{fig10}. Satellite $\K $ with pattern $\K _{P}$ and companion $\K _{C}$ admits a banded link diagram that is shown on the right of Figure \ref{fig11}. The corresponding marked graph diagram with multiplicities of $h^{\perp }_{\K }$ in most of its regions is shown in Figure \ref{fig12}. 
\end{example}

\begin{figure}[h!]
\labellist
\normalsize \hair 2pt
\pinlabel $\mathcal{D}$ at 680 400
\pinlabel $\K _{P}$ at 600 70
\pinlabel $\K _{C}$ at 1520 70
\pinlabel $2$ at 1190 100
\pinlabel $2$ at 1180 700
\pinlabel $0$ at 1040 400
\pinlabel $0$ at 1310 400
\pinlabel $2$ at 910 400
\pinlabel $2$ at 920 550
\pinlabel $2$ at 1460 400
\pinlabel $2$ at 1450 550
\pinlabel $4$ at 1400 610
\pinlabel $4$ at 970 610
\pinlabel $2$ at 1180 400
\pinlabel $W$ at 40 700
\endlabellist
\begin{center}
\includegraphics[scale=0.17]{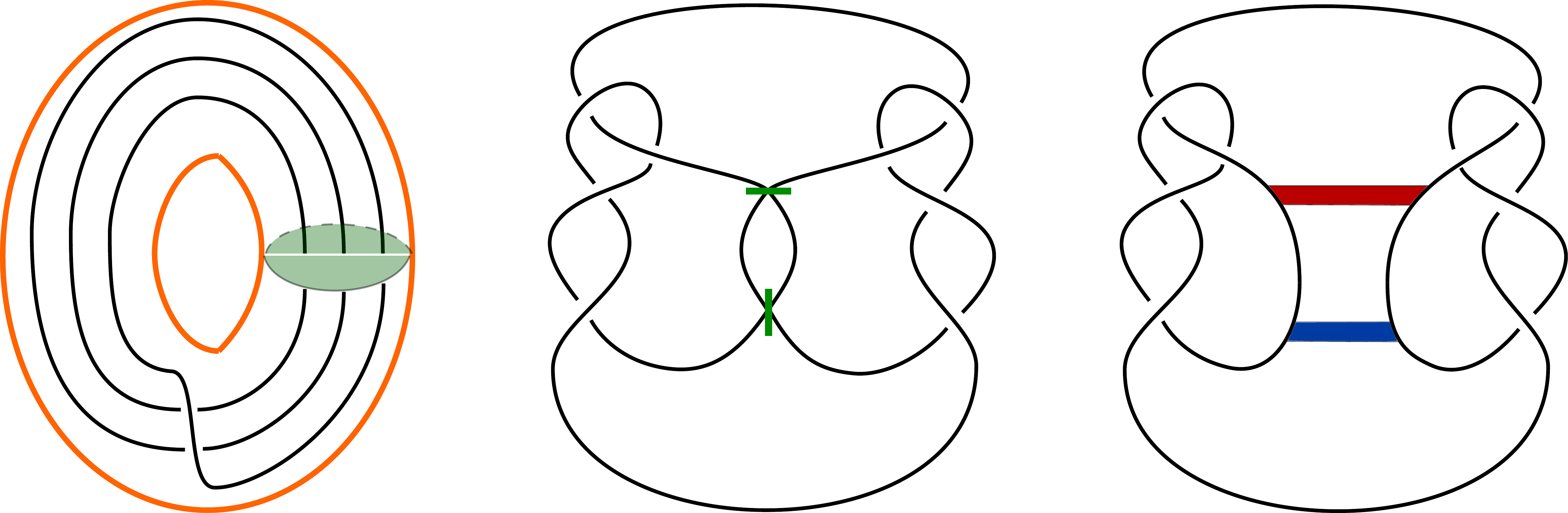}
\caption{Marked graph diagrams of the pattern $\K _P$ (left) and companion $\K _{C}$ (middle) of the satellite knot from Example \ref{ex3}. A link diagram of $\K _{C}$ in the normal form (right); the blue band lies below the 0-section and the red band lies above the 0-section.}
\label{fig10}
\end{center}
\end{figure}

\begin{figure}[h!]
\labellist
\normalsize \hair 2pt
\endlabellist
\begin{center}
\includegraphics[scale=0.18]{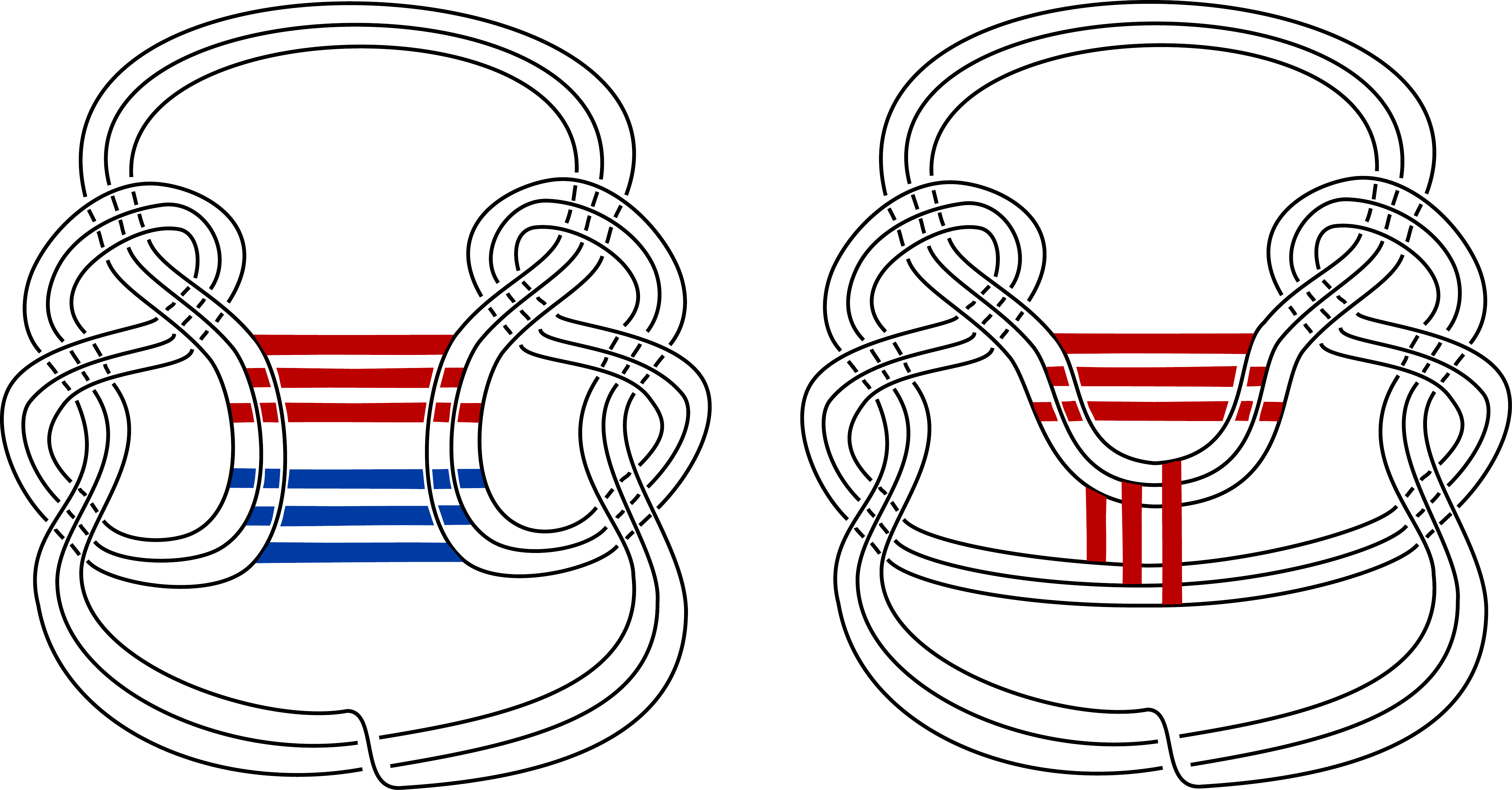}
\caption{A diagram in the normal form (left) and the corresponding banded link diagram (right) of the satellite knot $\K $ from Example \ref{ex3}. In the left diagram, the blue bands lie in $S^{4}_{-1}$, while the red bands lie in $S^{4}_{1}$. }
\label{fig11}
\end{center}
\end{figure}

The above example may lead to the following observations. 
Denote by $\Gamma $ (resp. $\Gamma _C$ and $\Gamma _{P}$) the marked graph diagrams of the satellite $\K $ (resp. companion $\K _C$ and pattern $\K _{P}$), described above. Let $h$ (resp. $h_C$ and $h_P$) be the hyperbolic decompositions of $\K $ (resp. $\K _C$ and $\K _P$), corresponding to these diagrams. The regions of $\Gamma $ consist of four different types:
\begin{enumerate}
\item regions in the complement of $f(W)$ correspond to the regions of $\Gamma _C$,
\item rectangular regions that correspond to the edges of $\Gamma _C$,
\item regions that correspond to the vertices of $\Gamma _C$,
\item regions that correspond to the regions of $\Gamma _P$.
\end{enumerate} Denote by $\omega $ the geometric winding of $\K _P$. Every region of $\Gamma _C$ induces one region of type (1), and every region of $\Gamma _P$ induces one region of type (4). There is only one region of $\Gamma $ that belongs to both type (1) and type (4); in the diagram on Figure \ref{fig12}, this is the lowest region with multiplicity $6$. Every edge of $\Gamma _C$ induces $(\omega -1)$ regions of type (2), and every vertex (either crossing or marked vertex) of $\Gamma _C$ induces $(\omega -1)^{2}$ regions of type (3). Every marked vertex of $\Gamma _C$ gives rise to $\omega $ marked vertices of $\Gamma $, and each of these vertices identifies two regions. It follows that $$p(\Gamma )=p(\Gamma _P)+p(\Gamma _C)-1+(\omega -1)e(\Gamma _C)+(\omega -1)^{2}\left (c(\Gamma _C)+v(\Gamma _C)\right )-(\omega -1)v(\Gamma _C)\;,$$ where $e(G)$ (resp. $c(G)$ and $v(G)$) denote the number of edges (resp. the number of crossings and marked vertices) of a marked graph $G$. Recall that the minimal number of vertices over all marked graph diagrams of a knotted surface $K$ is called \textbf{the \textit{ch}-index} of $K$ and denoted by $ch(K)$ \cite{YO}.  

Using the diagram of a satellite 2-knot, described above, we obtain an upper bound for its trunk. 

\begin{Theorem} Let $\K $ be a satellite 2-knot with companion $\K _C$ and pattern $\K _P$. Denote by $\omega $ the geometric winding of $\K _P$. Then 
\begin{xalignat}{1}\label{eq2}
\tr (\K )\leq \max \{\omega \tr (\K _C),\tr (\K _P)\}
\end{xalignat}
\end{Theorem}
\begin{proof}
Denote by $\Gamma $ (resp. $\Gamma _C$ and $\Gamma _{P}$) the marked graph diagrams of the satellite $\K $ (resp. companion knot $\K _C$ and pattern $\K _{P}$), obtained by the procedure, described above. Let $h$ (resp. $h_C$ and $h_P$) be the hyperbolic decompositions of $\K $ (resp. $\K _C$ and $\K _P$), corresponding to these diagrams. 

Recall the diffeomorphism $f\colon S^{2}\times D^{2}\to \nu (\K _C)$ maps the pattern $\K _P$ onto the satellite knot $\K $. Let $U$ be a region of $\Gamma $ of type (1) that corresponds to the region $U'$ of $\Gamma _C$, then the fiber of $h^{\perp}_{\K }$ above $U$ consists of $\omega $ copies of the fiber of $(h_{C}^{\perp})_{\K _C}$ over $U'$, and thus $m_{h^{\perp}_{\K }}(U)=\omega \, m_{(h^{\perp}_{C})_{\K _C}}(U')$. The multiplicity of $h^{\perp}_{\K }$ in a region of type (4) agrees with the multiplicity of $(h^{\perp}_{P})_{\K _P}$ in its corresponding region of $\Gamma _{P}$. 

Let $e$ be an edge of $\Gamma _C$ that separates two regions $U_1$ and $U_2$ of $\Gamma _C$. Then $e$ gives rise to $(\omega -1)$ regions of type (2), and multiplicities of $h^{\perp}_{K}$ in those regions interpolate between $\omega m_{(h^{\perp}_{C})_{\K _C}}(U_1)$ and $\omega m_{(h^{\perp}_{C})_{\K _C}}(U_2)$ (where multiplicity in each successive region jumps by 2). 

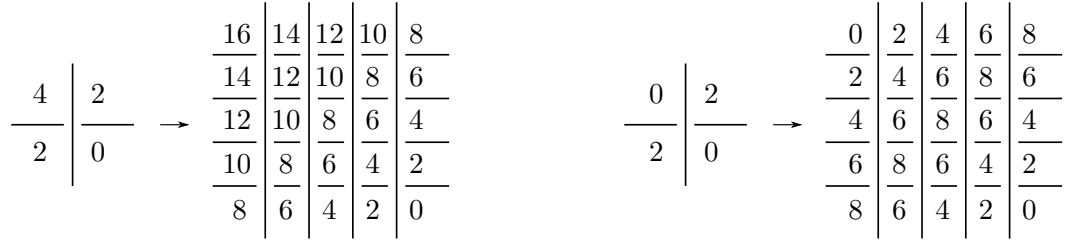
\begin{figure}[h!]
\labellist
\normalsize \hair 2pt
\pinlabel $2$ at 100 300
\pinlabel $4$ at 100 500
\pinlabel $2$ at 300 500
\pinlabel $0$ at 300 300
\pinlabel $8$ at 790 100
\pinlabel $6$ at 950 100
\pinlabel $4$ at 1100 100
\pinlabel $2$ at 1250 100
\pinlabel $0$ at 1400 100
\pinlabel $10$ at 780 260
\pinlabel $8$ at 950 260
\pinlabel $6$ at 1100 260
\pinlabel $4$ at 1250 260
\pinlabel $2$ at 1400 260
\pinlabel $12$ at 780 410
\pinlabel $10$ at 950 410
\pinlabel $8$ at 1100 410
\pinlabel $6$ at 1250 410
\pinlabel $4$ at 1400 410
\pinlabel $14$ at 780 560
\pinlabel $12$ at 950 560
\pinlabel $10$ at 1100 560
\pinlabel $8$ at 1250 560
\pinlabel $6$ at 1400 560
\pinlabel $16$ at 780 710
\pinlabel $14$ at 950 710
\pinlabel $12$ at 1100 710
\pinlabel $10$ at 1250 710
\pinlabel $8$ at 1400 710

\pinlabel $2$ at 2230 300
\pinlabel $0$ at 2230 500
\pinlabel $2$ at 2420 500
\pinlabel $0$ at 2420 300
\pinlabel $8$ at 2920 100
\pinlabel $6$ at 3070 100
\pinlabel $4$ at 3220 100
\pinlabel $2$ at 3370 100
\pinlabel $0$ at 3520 100
\pinlabel $6$ at 2920 260
\pinlabel $8$ at 3070 260
\pinlabel $6$ at 3220 260
\pinlabel $4$ at 3370 260
\pinlabel $2$ at 3520 260
\pinlabel $4$ at 2920 410
\pinlabel $6$ at 3070 410
\pinlabel $8$ at 3220 410
\pinlabel $6$ at 3370 410
\pinlabel $4$ at 3520 410
\pinlabel $2$ at 2920 560
\pinlabel $4$ at 3070 560
\pinlabel $6$ at 3220 560
\pinlabel $8$ at 3370 560
\pinlabel $6$ at 3520 560
\pinlabel $0$ at 2920 710
\pinlabel $2$ at 3070 710
\pinlabel $4$ at 3220 710
\pinlabel $6$ at 3370 710
\pinlabel $8$ at 3520 710
\endlabellist
\begin{center}
\includegraphics[scale=0.12]{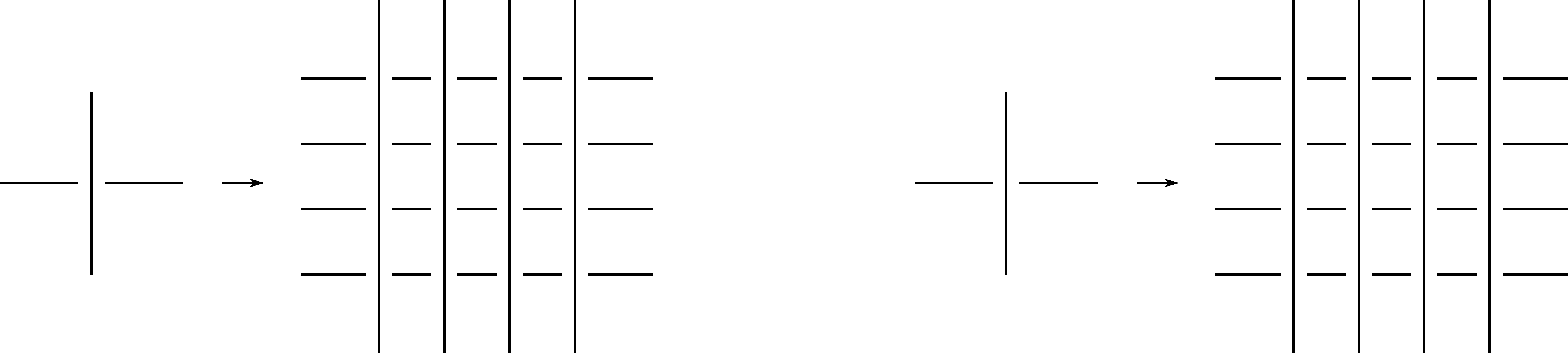}
\caption{A bunch of crossings of $\Gamma $, induced by a crossing $v$ of $\Gamma _C$: if $v$ is a fold crossing (left) and if $v$ is a branch point (right), in the  case $\omega =4$.}
\label{fig15}
\end{center}
\end{figure}

Any vertex $v$ of $\Gamma _C$ is incident to four edges of $\Gamma _C$ and gives rise to $(\omega -1)^{2}$ regions of type (3). The multiplicities of $h^{\perp}_{\K}$ in these regions depend on the type of the vertex $v$, see Figure \ref{fig15}. In case $v$ is a fold crossing, the multiplicities in the regions of type (3) interpolate between the multiplicities of $h^{\perp}_{\K}$ in the regions of type (2) corresponding to the edges incident at $v$ (where multiplicity in each successive region jumps by 2). In case $v$ is a branch point, the highest multiplicity of $h^{\perp}_{\K}$ in a region of type (3), corresponding to $v$, coincides with the highest multiplicity of $h^{\perp}_{\K}$ in a region of type (1) that comes from a region of $\Gamma _C$, incident to $v$. 

It follows from the above discussion that the highest multiplicity of $h^{\perp}_{\K }$ is either attained in a region of type (1) or in a region of type (4) and therefore $$\tr (\Gamma )=\max \{\omega \tr (\Gamma _C),\tr (\Gamma _P)\}\;.$$ By choosing the diagrams $\Gamma _C$ and $\Gamma _P$ so that $\tr (\Gamma _C)=\tr (\K _C)$ and $\tr (\Gamma _P)=\tr (\K _P)$, we obtain the desired inequality.   
\end{proof}

\begin{remark} If the geometric winding of $\K _P$ equals 1, then the satellite with pattern $\K _P$ and companion $\K _C$ is in fact the connected sum $\K _{P}\# \K _{C}$. In the special case when $\K _P$ is unknotted, we have $\K _{P}\# \K _{C}=\K _{C}$ and since $\tr (\K _C)\geq 2$, the inequality \eqref{eq2} becomes an equality. 
\end{remark}

 In a recent paper \cite{FH}, Freedman and Hillman use the satellite construction together with some intricate topological machinery to show that there exist $n$-dimensional knots in $\RR ^{n}$ of arbitrarily large width for each $n\geq 1$. Their definition of width is slightly different, but conceptually similar to our trunk invariant in dimension 2. In order to distinguish between Takeda's width and the width defined by Freedman and Hillman, we will denote the latter by $w_{FH}$.  

\begin{definition} \label{defF}\cite{FH} Given a smooth embedding $K\colon S^{2}\hookrightarrow \RR^4$, let $\pi \colon \RR ^{4}\to \RR ^{2}$ be any composition $\RR ^{4}\stackrel{d}{\rightarrow}\RR ^{4}\stackrel{p}{\rightarrow} \RR ^{2}$, where $d$ is any diffeomorphism and $p$ denotes the projection onto the last 2 coordinates. The \textbf{width} of $K$ is denoted by $w_{FH}(K)$ and defined as $$w_{FH}(K)=\min _{\pi }\left \{\max \left \{|K(S^{2})\cap \pi ^{-1}(p)|\, \colon \, p\in \RR ^{n}\textrm{ a regular value of the composition }\pi \circ K\right \}\right \}\;, $$ where the minimum is taken over all product projections $\pi $ specified above. 
\end{definition}

Comparing this definition with Takeda's definition of width (see Remark \ref{rem1} and Definition \ref{defT}), we may conclude the following: 

\begin{lemma} \label{lemma8} For any 2-knot $\K $, we have $$w_{FH}(\K )\leq w(\K )\leq \tr (\K )\;.$$
\end{lemma}
\begin{proof} Suppose $K\colon S^{2}\hookrightarrow \RR ^{4}$ is a smooth embedding. Any orthogonal projection $\pi \colon \RR ^{4}\to \RR ^{2}$ that is generic with respect to $K$, possibly precomposed by an isotopy of $K(S^{2})$, defines a product projection as specified in Definition \ref{defF}. Therefore $w_{FH}(\K )\leq w(\K )$, while the second inequality is provided by Corollary \ref{cor2}.
\end{proof}

Using the above relationship, we may establish the following. 

\begin{proposition} There exist knotted spheres in $S^{4}$ with arbitrarily large trunk. There exist knotted spheres for which the difference between the left-hand side and the right-hand side of the inequality \eqref{eq1} is arbitrarily large. 
\end{proposition}
\begin{proof} By \cite[Theorem 4]{FH}, there exist smooth knots $K\colon S^{2}\hookrightarrow \RR ^{4}$ with arbitrarily large width $w_{FH}(\K )$. It follows by Lemma \ref{lemma8} that the corresponding knots in $S^{4}$ have arbitrarily large trunk. To verify the second statement, suppose that a knotted sphere $\K $ has $\tr (\K )=d$ for some even $d\gg 2$. Then Corollary \ref{cor3} implies that for any marked graph diagram $\Gamma $ of $\K $, there exist regions $U_i$ of $\Sigma \backslash \Gamma $ with $m_{h_{\K }^{\perp }}(U_i)=2+2i$ for $i=1,2,\ldots ,\frac{d-4}{2}$. It follows that $\lay (\Gamma )\geq 2p(\Gamma )+\tr (\Gamma )-2+\frac{1}{4}(d^{2}-6d+8)$ for any marked graph diagram $\Gamma $. Choose $\Gamma $ for which $\lay (\Gamma )=\lay (\K )$ and we obtain $$\lay (\K )-\left (2p(\K )+\tr (\K )-2\right )\geq \lay (\Gamma )-\left (2p(\Gamma )+\tr (\Gamma )-2\right )\geq \frac{1}{4}\left (d^{2}-6d+8\right )\;.$$  
\end{proof}

The first example of a 2-knot in $S^4$ that becomes unknotted when connect summing with a standard real projective plane was found by Viro \cite{VI}. Using the satellite construction, Kim constructed another infinite family of such examples which are not ribbon 2-knots \cite{SK}. His examples, together with Freedman and Hillman's results, provide the following. 

\begin{Theorem} There exist knotted surfaces $\K _{1}$ and $\K _{2}$ in $S^{4}$ for which the difference between the right-hand side and the left-hand side of the inequality $$\tr (\K _{1}\# \K _{2})\leq \max \{\tr (\K _{1}),\tr (\K _{2})\}$$ from Proposition \ref{prop4} is arbitrarily large.   
\end{Theorem}
\begin{proof} In \cite{SK}, the author constructed an infinite number of satellite 2-knots which become unknotted by connected summing with a standard real projective plane $\PP ^{2}$. Specifically, these are obtained as $2n$-cables of the 2-twist spin of any 2-bridge knot. For the definition of twist-spinning, see Zeeman \cite{ZE}. Let $k$ be a 2-bridge knot, denote by $\tau _{2}(k)$ the 2-twist spin of $k$ and by $\K _{n}$ its $2n$-cable, where $n$ is a positive integer. By \cite{FH}, a 2-twist spin of a nontrivial classical knot has a positive homological width $w_{H}(\tau _{2}(k))$ and by \cite[Theorem 2]{FH}, its $2n$-cable has width $w_{FH}(\K _{n})\geq 2n w_{H}(\tau _{2}(k))\geq 2n$. It follows by Lemma \ref{lemma8} that $\tr (\K _{n})\geq 2n$. 

The standard real projective plane $\PP ^{2}$ admits a marked graph diagram $\Gamma $ with $\tr (\Gamma )=2$, see Figure \ref{fig4}. It follows that $\tr (\PP ^{2})\leq 2$ and since $w(\PP ^{2})=2$ by \cite{TA}, Co\-ro\-lla\-ry \ref{cor2} implies $\tr (\PP ^{2})=2$. By \cite[Theorem 2.9]{SK} we have $\K _{n}\# \PP ^{2}=\PP ^{2}$, therefore $\max \{\tr (\K _{n}),\tr (\PP ^{2})\}-\tr (\K _{n}\# \PP ^{2})=2n-2$. 
\end{proof}

\subsection{Directions for further study} \label{subs42}

\begin{enumerate}
\item The width of classical knots is closely related with the bridge number. Bridge decompositions of knotted surfaces were studied by Meier and Zupan, who obtained several results about the bridge number of surfaces inside the 4-sphere and in other 4-manifolds \cite{MZ,MZ1}. An important goal would be to understand the relationship between the bridge number, the layering and the partition number of an embedded surface. 

\item In the same direction, we would like to investigate the connection between the flattenings of surfaces and trisections of surfaces, presented in \cite{MZ,MZ1}. It is an interesting question whether some information, obtained by flattening a surface $\K $, could actually be read from a suitable trisection diagram of $\K $. 

\item It might be fruitful to explore the shape of regions in a surface diagram. Are there typical shape structures occuring in the diagrams of some families of knotted surfaces? What does the shape of regions in a diagram of a surface $\K $ tell us about its properties? What can we learn by investigating the values of the shape invariant $s(\K )$?

\item Several results concerning the bridge number and width of classical satellite knots have been established \cite{SC1,SC2,ZU}. As we demonstrate in Subsection \ref{subs41}, fla\-tte\-nings offer a suitable way to study the layering and other invariants of satellite 2-knots yet unexplored. 

\item Our construction could be generalized to surfaces in an arbitrary four manifold. Flattenings thus obtained could increase our means of presenting and understanding embedded surfaces. 

\end{enumerate}

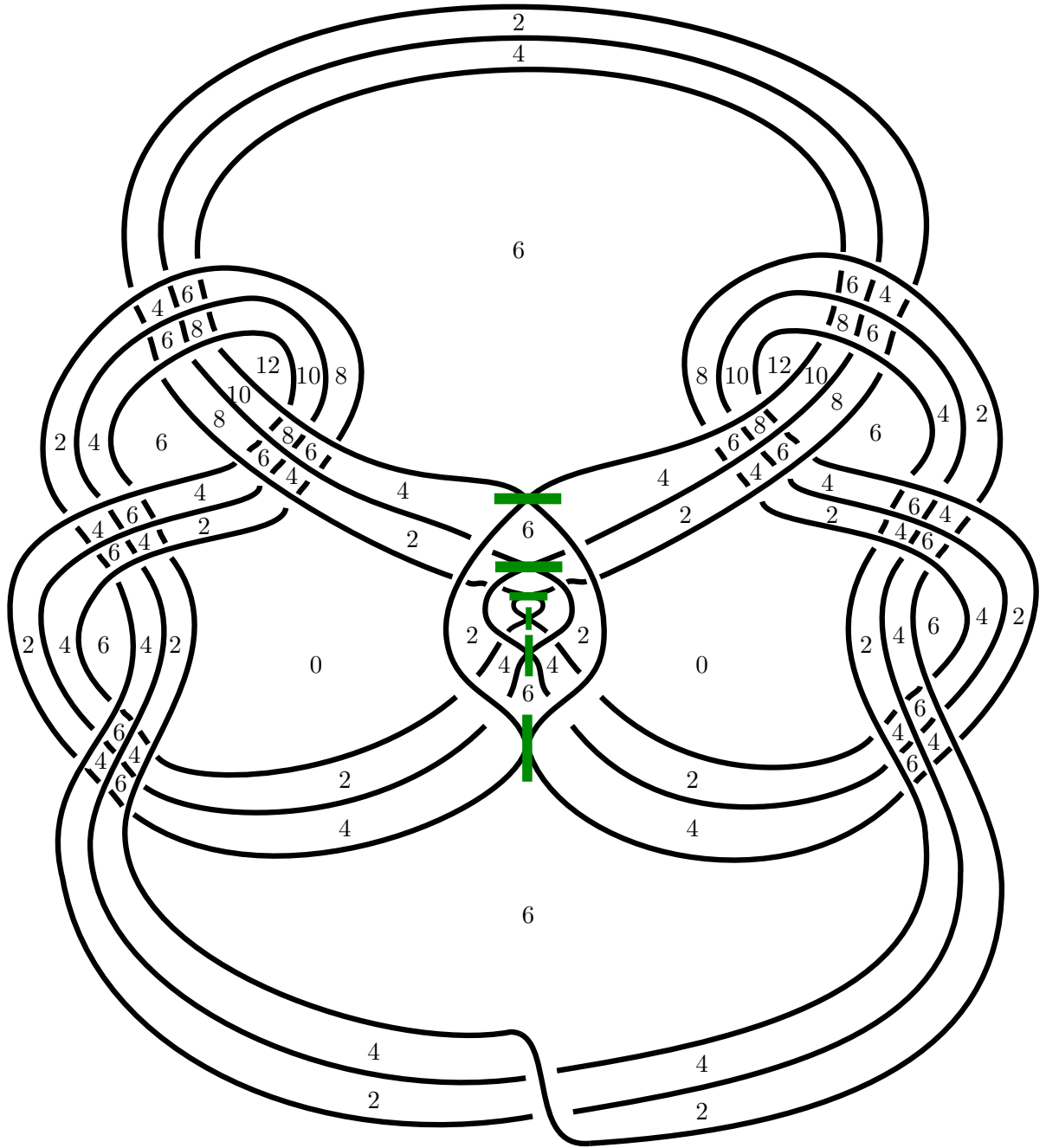
\begin{figure}[h!]
\labellist
\normalsize \hair 2pt
\pinlabel $6$ at 540 240
\pinlabel $2$ at 380 48
\pinlabel $4$ at 380 98
\pinlabel $2$ at 720 36
\pinlabel $4$ at 720 86
\pinlabel $6$ at 530 930
\pinlabel $4$ at 530 1134
\pinlabel $2$ at 530 1166
\pinlabel $0$ at 320 500
\pinlabel $0$ at 720 500
\pinlabel $4$ at 350 330
\pinlabel $4$ at 710 330
\pinlabel $2$ at 350 380
\pinlabel $2$ at 710 380
\pinlabel $6$ at 100 520
\pinlabel $6$ at 160 730
\pinlabel $6$ at 960 540
\pinlabel $6$ at 900 740
\pinlabel $12$ at 800 810
\pinlabel $12$ at 270 810
\pinlabel $2$ at 22 520
\pinlabel $4$ at 60 520
\pinlabel $4$ at 144 520
\pinlabel $2$ at 174 520
\pinlabel $2$ at 1048 550
\pinlabel $4$ at 1010 550
\pinlabel $2$ at 55 730
\pinlabel $4$ at 90 730
\pinlabel $2$ at 1010 760
\pinlabel $4$ at 970 760
\pinlabel $4$ at 924 530
\pinlabel $2$ at 890 520
\pinlabel $4$ at 200 678
\pinlabel $2$ at 204 643
\pinlabel $4$ at 850 688
\pinlabel $2$ at 855 655
\pinlabel $4$ at 680 694
\pinlabel $2$ at 703 655
\pinlabel $4$ at 410 680
\pinlabel $2$ at 420 630
\pinlabel $8$ at 220 755
\pinlabel $10$ at 240 780
\pinlabel $10$ at 312 800
\pinlabel $8$ at 346 800
\pinlabel $8$ at 720 800
\pinlabel $10$ at 756 800
\pinlabel $4$ at 156 870
\pinlabel $6$ at 166 836
\pinlabel $6$ at 187 884
\pinlabel $8$ at 197 848
\pinlabel $6$ at 876 894
\pinlabel $8$ at 866 854
\pinlabel $4$ at 910 884
\pinlabel $6$ at 897 844
\pinlabel $10$ at 838 800
\pinlabel $8$ at 860 773
\pinlabel $6$ at 540 470
\pinlabel $6$ at 540 640
\pinlabel $2$ at 482 530
\pinlabel $2$ at 597 530
\pinlabel $4$ at 515 500
\pinlabel $4$ at 565 500
\pinlabel $4$ at 94 640
\pinlabel $6$ at 111 616
\pinlabel $6$ at 130 655
\pinlabel $4$ at 142 627
\pinlabel $4$ at 972 656
\pinlabel $6$ at 940 664
\pinlabel $4$ at 926 638
\pinlabel $6$ at 954 628
\pinlabel $4$ at 960 420
\pinlabel $6$ at 946 455
\pinlabel $6$ at 938 397
\pinlabel $4$ at 923 430
\pinlabel $4$ at 132 406
\pinlabel $6$ at 116 430
\pinlabel $6$ at 118 377
\pinlabel $4$ at 97 400
\pinlabel $6$ at 266 714
\pinlabel $4$ at 295 695
\pinlabel $8$ at 291 738
\pinlabel $6$ at 315 721
\pinlabel $6$ at 753 730
\pinlabel $4$ at 776 700
\pinlabel $8$ at 780 748
\pinlabel $6$ at 804 721
\endlabellist
\begin{center}
\includegraphics[scale=0.43]{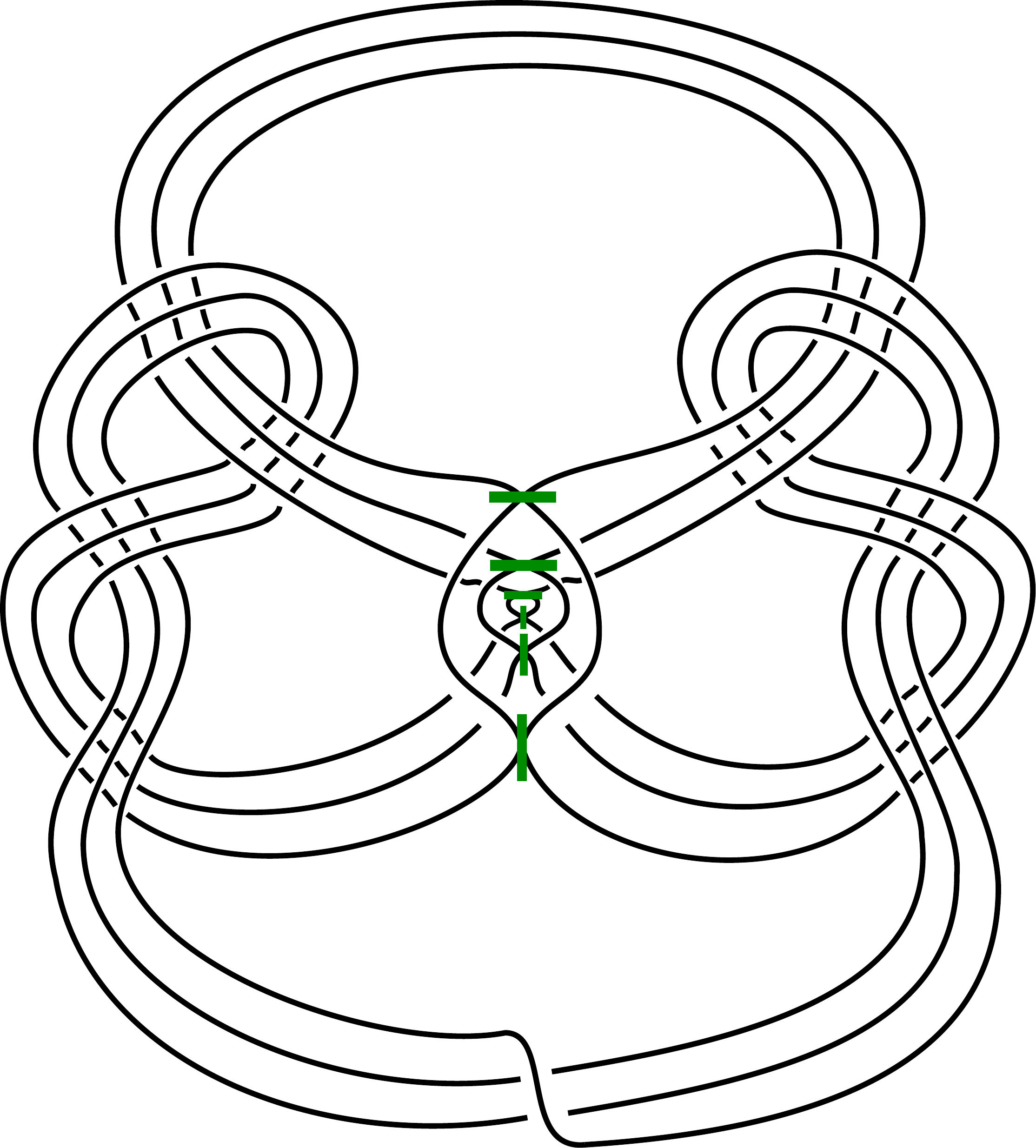}
\caption{A marked graph diagram of the satellite knot $\K $ from Example \ref{ex3}.}
\label{fig12}
\end{center}
\end{figure}

\section*{Data availability statement}

Data sharing is not applicable.

\end{document}